\documentclass[10pt]{article}       
\usepackage{geometry}               
\usepackage{graphicx}
\usepackage{caption}
\usepackage{subcaption}
\usepackage{mathrsfs}
\usepackage{comment}
\usepackage{amsmath} 
\usepackage{amssymb}  
\usepackage{amsthm}  
\usepackage{bm}
\usepackage{color}
\usepackage{array}
\usepackage{cite}
\usepackage{hyperref}
\usepackage{comment}

\newtheorem{proposition}{Proposition}
\newtheorem{defn}{Definition}

\newtheorem{lemma}{Lemma}
\newtheorem{theorem}{Theorem}
\newtheorem{remark}{Remark}

\numberwithin{equation}{section}

\newcommand{\R}{\mathbb{R}}

\def\pa{\partial}
\def\na{\nabla}
\def\eps{\varepsilon}
\def\div{\mathrm{div}\, }

\DeclareMathOperator*{\argmin}{arg\,min}

\title{
Diffusion-aggregation equations and 
 volume-preserving mean curvature flows}

\newcommand{\footremember}[2]{%
    \footnote{#2}
    \newcounter{#1}
    \setcounter{#1}{\value{footnote}}%
}

\author{
  Jiwoong Jang \footremember{trailer}{Department of Mathematics, University of Maryland-College Park (jjang124@umd.edu)},
  Antoine Mellet \footremember{alley}{Department of Mathematics, University of Maryland-College Park (mellet@umd.edu).  Partially supported by NSF Grant DMS-2307342.
}
  }

\date{}
\providecommand{\keywords}[1]{\textbf{Key words.} #1}
\providecommand{\MSC}[1]{\textbf{MSC codes.} #1}

\begin{document}
\maketitle

\pagestyle{myheadings}
\thispagestyle{plain}

\begin{abstract}

The Patlak-Keller-Segel system of equations (PKS) is a classical example of aggregation-diffusion equation. It describes the aggregation of some organisms via chemotaxis, limited by some nonlinear diffusion. 
It is known that for some choice of this nonlinear diffusion, the PKS model asymptotically leads to phase separation and mean-curvature driven free boundary problems. 
In this paper, we focus on the Elliptic-Parabolic PKS model 
and we obtain the first unconditional convergence result in dimension $2$ and $3$ towards the volume preserving mean-curvature flow.
This work builds up on previous results that were obtained under the assumption that phase separation does not cause energy loss in the limit. 
In order to avoid this assumption, we  rely on Brakke type formulation of the mean-curvature flow 
and a reinterpretation of the problem as an Allen-Cahn equation with a nonlocal forcing term.
\end{abstract}

\keywords{Chemotaxis, Singular limit, Mean curvature flow, Varifolds.}

\vspace{0.2cm}

\MSC{35A15, 35K55, 53E10.}

\section{Introduction}\label{sec:introduction}

\subsection{The Patlak-Keller-Segel model}\label{subsec:PKS}
The classical \emph{parabolic-parabolic Patlak-Keller-Segel (PKS) system} for chemotaxis with nonlinear diffusion can be written as:
\begin{align}\label{eq:elliptic-parabolic}
\begin{cases}
\alpha\pa_t \rho -\mathrm{div}(\rho\nabla f'(\rho))+\chi\mathrm{div}(\rho\nabla\phi)=0, \\
\beta\partial_t\phi-\eta\Delta\phi=\rho-\sigma\phi.
\end{cases}
\end{align}
In this model, $\rho(t,x)$ is the density function of some organisms (such as bacteria or other types of cells) and $\phi(t,x)$ denotes the concentration of a chemoattractant. The organisms are advected toward  regions of higher concentration of the chemoattractant with proportionality constant $\chi>0$. The parameter $\eta>0$ denotes the diffusivity of the chemoattractant and the constant $\sigma>0$ denotes the destruction rate of the chemoattractant.

We note that the relaxation time of the organisms (the parameter $\alpha$ in \eqref{eq:elliptic-parabolic}) and that of the chemoattractant (the parameter $\beta$) have no reason to be equal and can in fact differ by several orders of magnitude.
In this paper,  we will consider the particular case  
$$\alpha=0,\quad \beta=1$$
which leads to the \emph{elliptic-parabolic PKS model} (note that
formally at least, the results presented in this paper are generalizable to regimes in which the organisms reach their equilibrium much faster than the chemoattractant diffuses, though the rigorous analysis is much  more challenging in that case).
The  opposite regime $\alpha=1$, $\beta=0$ yields the popular \emph{parabolic-elliptic PKS model} in which  the chemoattractant reaches its equilibrium much faster than the organisms. The well-posedness of the PKS model has received a lot of attention for a wide range of nonlinearity $f(\rho)$ (see \cite{JL92,HV96,M17,KMW24(2)} and the references therein).
The nonlinear diffusion term in \eqref{eq:elliptic-parabolic} takes into account the natural repulsive forces acting on the organisms and plays a very important role in our paper.
A classical fact about the PKS model is  that linear diffusion (which corresponds to $f(\rho) = \rho \log\rho$ with our notations) can lead to finite time blow up of the solution ($\lim_{t\to t^*} \| \rho(t)\|_{L^\infty} = +\infty$), a phenomena which describes the concentration of the organisms \cite{BBR11}. 
We are not interested in this phenomena 
and we will consider models with stronger repulsion  for large values of $\rho$ which do not lead to finite time blow-up.
A common choice for this nonlinearity $f$ is the power law 
$$f_m(\rho)=\frac{\rho^m}{m-1}, \qquad  m>1,$$ 
which penalizes high densities of the organisms.
For $m>m_c := 2-\frac{2}{d}$, it has been shown that concentration cannot occur and that solutions remain bounded in $L^\infty$ uniformly in time.
Another interesting phenomena occurs when $m>2$: 
Several recent work \cite{KMW24,KMW24(2),M23} have shown that in this case,
not only do solutions remain bounded for all time, but they experience phase separation phenomena at appropriate time scale.
 
\medskip

A possible microscopic interpretation of the model with large $m$ is that  the organisms have some finite size and that overlapping  is permitted, but heavily penalized. 
In the limit $m\to\infty$,  we get a hard sphere - or incompressible -  model in which overlapping is prohibited and the density must satisfy an a priori upper bound.
This corresponds to the nonlinearity
 $$f_{\infty}(\rho)=\begin{cases}
 0 & \mbox{ for } \rho\leq1,\\
 \infty & \mbox{ for } \rho>1.
 \end{cases}
 $$
The corresponding equation enforces the congestion constraint $\rho\leq1$ and requires the introduction of a pressure term (Lagrange multiplier) in the equation.
Since $f_\infty$ is convex, we can define its subdifferential  $\pa f_\infty(\rho)$ and rewrite the first equation in \eqref{eq:elliptic-parabolic} as
$$
\alpha\pa_t \rho -\mathrm{div}(\rho\nabla p)+\chi\mathrm{div}(\rho\nabla\phi)=0, \qquad p\in \pa f_\infty(\rho),
$$
where the condition $p\in \pa f_\infty(\rho)$ can also be written as $\rho \leq 1$, $ p\geq 0$ and $ p(1-\rho)=0$.
We  stress out the fact that this singular limit $m\to\infty$ has the double effect of enforcing the constraint $\rho\leq1$ and eliminating all diffusion when $\rho<1$.
In this paper we will take a nonlinearity $f(\rho)$ which  combines some nonlinear diffusion for small value of $\rho$ together with the hard-sphere constraint $\rho\leq1$: From now on, we take
\begin{equation}\label{eq:f0}
f(\rho):=A f_2(\rho) + f_\infty(\rho)=
\begin{cases}
A \rho^2 & \mbox{ if } \rho \in[0,1] \\
+\infty & \mbox{ otherwise} 
\end{cases}
,\qquad A>0.
\end{equation}

Our goal is to show that phase separation takes place,  resulting in the formation of high and low density regions separated by a sharp interface, and  to characterize the evolution of that  interface. 
A similar study was undertaken in several recent papers (see \cite{KMW24,KMW24(2),M23,MR24}) 
for the parabolic-elliptic and elliptic-parabolic PKS model and recently extended to the full parabolic-parabolic model in \cite{R25}.
As in these papers, we assume that a large number of individuals are observed from far away:
We introduce a parameter $\varepsilon>0$ that quantifies the large (initial) mass
\[
\int\rho_{\textrm{in}}(x)dx=\varepsilon^{-d} \gg 1,
\]
and rescale the problem as follows: 
\[
x\mapsto \varepsilon x,\qquad t\mapsto \eta \beta^{-1}\varepsilon^2 t.
\]
In these new (macroscopic)  variables,  \eqref{eq:elliptic-parabolic} becomes:
\begin{align}\label{eq:elliptic-parabolic-rescaled0}
\begin{cases}
\frac{\eta \alpha}{\beta}\pa_t\rho^\eps
-\mathrm{div}(\rho^{\varepsilon}\nabla p^\eps )+\chi\mathrm{div}(\rho^{\varepsilon}\nabla\phi^{\varepsilon})=0,  \quad p^\eps \in \pa f(\rho^{\varepsilon})\quad &\text{in}\quad(0,\infty)\times\Omega,\\
\partial_t\phi^{\varepsilon}- \Delta\phi^{\varepsilon}=\eta^{-1} \varepsilon^{-2}(\rho^{\varepsilon}-\sigma\phi^{\varepsilon}), \quad &\text{in}\quad(0,\infty)\times\Omega.
\end{cases}
\end{align}
Here $\Omega$ is a fixed domain (which has macroscopic size $\sim 1$).
This system will be supplemented with no-flux boundary conditions on $\partial\Omega$ and with the initial conditions by given profiles $\rho_{\textrm{in}}^{\varepsilon},\phi_{\textrm{in}}^{\varepsilon}$ now with normalized initial mass
\begin{equation}\label{eq:mass}
\int_\Omega \rho(x,t) \, dx = 
\int_{\Omega}\rho_{\textrm{in}}^{\varepsilon}(x)dx=1.
\end{equation}
We note that by redefining $\sqrt{\frac{\eta}{\sigma}} \eps \mapsto \eps$, $\sigma\phi\mapsto \phi$, $\frac{\sigma}{\chi} f \mapsto f$, we can rewrite \eqref{eq:elliptic-parabolic-rescaled0} as
\begin{align}\label{eq:elliptic-parabolic-rescaled1}
\begin{cases}
\alpha'\pa_t \rho^\eps
-\mathrm{div}(\rho^{\varepsilon}\nabla p^\eps)+ \mathrm{div}(\rho^{\varepsilon}\nabla\phi^{\varepsilon})=0,\qquad p^\eps\in\pa f (\rho^{\varepsilon}) \quad &\text{in}\quad(0,\infty)\times\Omega,\\
\partial_t\phi^{\varepsilon}- \Delta\phi^{\varepsilon}=  \varepsilon^{-2}(\rho^{\varepsilon}- \phi^{\varepsilon}), \quad &\text{in}\quad(0,\infty)\times\Omega.
\end{cases}
\end{align}
thus eliminating all the parameters with the exception of the relaxation time 
$$ \alpha'= \frac{\chi \eta}{\sigma \beta} \alpha.$$
As mentioned at the beginning of this introduction, this paper focuses on the case $\alpha=0$ which describes a situation in which the organisms adapt instantly to the distribution of the chemoattractant $\phi$.
The equation for $\rho^\eps$ becomes elliptic and the effect of the initial condition $\rho_{in}$ is reduced to the mass condition \eqref{eq:mass}.
The goal of this paper is thus to investigate the asymptotic behavior as $\eps\to0$ of the solution of:
\begin{align}\label{eq:elliptic-parabolic-rescaled}
\begin{cases}
-\mathrm{div}(\rho^{\varepsilon}\nabla p^\eps)+ \mathrm{div}(\rho^{\varepsilon}\nabla\phi^{\varepsilon})=0,\quad p^\eps\in\pa f (\rho^{\varepsilon}), \quad \int_\Omega \rho^\eps(t,x) \, dx =1 \quad &\text{in}\quad(0,\infty)\times\Omega,\\
\partial_t\phi^{\varepsilon}- \Delta\phi^{\varepsilon}=  \varepsilon^{-2}(\rho^{\varepsilon}- \phi^{\varepsilon}), \quad &\text{in}\quad(0,\infty)\times\Omega.
\end{cases}
\end{align}
with the boundary conditions ($n$ being the outer unit normal vector of $\partial\Omega$)
\begin{equation}\label{eq:bc}
-\rho^\eps\na (p^\eps -\phi^\eps)\cdot n=0, \qquad \na\phi^\eps\cdot n=0 \qquad \mbox{ on } (0,\infty)\times\pa\Omega
\end{equation}
and the initial condition
$$
\phi^\eps(0,x) = \phi_{in}(x).
$$

This limit $\eps\to 0$  was previously studied in \cite{MR24} where it was proved that for well prepared initial condition, both $\phi^\eps$ and $\rho^\eps$ converge strongly to a characteristic function $\chi_{E(t)}$ and that the interface $\pa E(t)$ evolves according to the {\bf volume preserving mean-curvature flow}
\begin{align}\label{eq:VPMCF}
V=-\kappa+\Lambda\quad\text{on}\ \partial E(t)\cap \Omega,\qquad|E(t)|=1\quad\text{for }t>0, 
\end{align}
where  $V=V(t,x)$ is the outward normal velocity at a point $x$ of the boundary $\partial E(t)$, $\kappa=\kappa(t,x)$ is the mean curvature of $\partial E(t)$ and  $\Lambda=\Lambda(t)$ is a Lagrange multiplier associated to the constraint $|E(t)|=1$.

However, in \cite{MR24} equation \eqref{eq:VPMCF} was derived under an assumption on the convergence of the energy (see \eqref{eq:energy-assumption}) which we do not know how to establish a priori.
\emph{The goal of this paper is to establish a similar result rigorously without this assumption and thus get the first unconditional convergence result in this direction.}

A key idea in the paper - which we explain in the next section -  is that  \eqref{eq:elliptic-parabolic-rescaled} can be recast as a nonlocal Allen-Cahn equation. We will thus prove our result by adapting strategies developed for volume preserving Allen-Cahn equations to our model. This requires some delicate estimates that we derive by using the particular form of the nonlinearity \eqref{eq:f0}.

\medskip


\subsection{A non-local Allen-Cahn equation}\label{subsec:energy}

For a given $\phi$, the equation for the density $\rho$ describes the instantaneous relaxation of the organisms toward 
an equilibrium that balances the repulsive effect of the pressure $f(\rho)$ and the attractive effect of the chemo-attractant. 
It should be interpreted as describing the long time asymptotic of the corresponding evolution equation and is thus naturally  supplemented with non-negativity and mass constraint, which can be summarized by requiring that $\rho \in \mathcal{P}_{\textrm{ac}}(\Omega)$. Here $\mathcal{P}_{\textrm{ac}}(\Omega)$ denotes the set of probability measures  on $\Omega$ that are absolutely continuous with respect to the $d$-dimensional Lebesgue measure on $\Omega$  (a measure in $\rho\in\mathcal{P}_{\textrm{ac}}(\Omega)$ will always be identified with its density in this paper).
For a given potential $\phi=\phi^\eps(t,\cdot)$, the density $\rho=\rho^\eps (t,\cdot)$ is thus solution of 
\begin{equation}\label{eq:stat}
\begin{cases}
-\mathrm{div}(\rho \nabla p)+ \mathrm{div}(\rho \nabla\phi )=0, \quad p \in \pa f(\rho) &\text{in}\quad \Omega,\qquad \int_\Omega \rho(x)\, dx=1\\
-\rho [ \nabla p-\na \phi] \cdot n = 0 & \text{on} \quad  \pa \Omega.
 \end{cases}
\end{equation}
Solutions of \eqref{eq:stat} may not be unique, but the corresponding minimization problem 
\begin{equation}\label{eq:rhophi}
\rho_\phi = \mathrm{argmin}\,\left\{ \int_{\Omega}f(\rho)-\rho\phi\, dx\,;\, \rho \in \mathcal{P}_{\text{ac}}(\Omega)\right\}
\end{equation}
has a unique solution. This is immediate since when $f$ is given by \eqref{eq:f0}, we can write \eqref{eq:rhophi} as the minimization of the uniformly convex energy $\int_{\Omega}A \rho^2-\rho\phi dx$ on the convex set $K:=\{\rho\in L^2(\Omega)\,;\, 0\leq \rho\leq 1\mbox{ in } \Omega, \; \int_\Omega \rho\, dx=1\}$. To ensure that $K$ is not empty, we will always assume that
\begin{align}\label{eq:Omega1}
|\Omega|>1.
\end{align}
To justify the choice of the energy minimizer among all possible solutions of \eqref{eq:stat},  we point out that when a linear diffusion $\nu\Delta \rho$ is added to \eqref{eq:stat} (modeling some noise in the behavior of the organisms), it has a unique solution which is  also the unique global minimizer. Because the global minimizer is stable with respect to the limit $\nu\to0$, our choice for $\rho_\phi$ can be seen as a consequence of a natural vanishing noise approximation.

Important facts about the  minimization problem \eqref{eq:rhophi} will be recalled in Proposition \ref{prop:well-posed-min-prob}, but for now, we just note that it allows us to define a map
$\phi \mapsto \rho_\phi$
which identifies the density distribution of the organisms for a given distribution of the chemo-attractant.
With this notation, we can rewrite \eqref{eq:elliptic-parabolic-rescaled} as a single (nonlocal) equation
\begin{align}\label{eq:phiepsilon}
\begin{cases}
\partial_t\phi^{\varepsilon}-\Delta\phi^{\varepsilon}=\varepsilon^{-2}\left(\rho_{\phi^{\varepsilon}}- \phi^{\varepsilon}\right) \quad &\text{in}\quad(0,\infty)\times\Omega,\\
\nabla\phi^{\varepsilon}\cdot n=0 \quad &\text{on}\quad(0,\infty)\times\partial\Omega, \\
\phi^{\varepsilon}(0,\cdot)=\phi^{\varepsilon}_{\textrm{in}} \quad &\text{in}\quad \Omega.
\end{cases}
\end{align}

\medskip


This definition of $\rho_\phi$ is also consistent with the natural energy structure of the PKS model. 
Indeed, we recall that   the Parabolic-Parabolic system \eqref{eq:elliptic-parabolic-rescaled1} is associated with the energy
$$
\mathscr E_\eps (\rho,\phi) := \frac{1}{\varepsilon}\int_{\Omega}\left(f(\rho)+\left(\frac12-A\right)\rho-\rho\phi+\frac{1}{2}\phi^2\right)dx+\frac{\varepsilon}{2}\int_{\Omega}|\nabla\phi|^2\, dx .
$$
Because of the mass conservation property, the term $\left(\frac12-A\right)\rho$ only changes the energy by a constant, but it will play an important  role later on. 
Not surprisingly, the Elliptic-Parabolic system \eqref{eq:elliptic-parabolic-rescaled} is in turn associated with the  energy functional
\begin{equation}\label{eq:energy2}
\mathscr{J}_{\varepsilon}(\phi) = 
\inf_{\rho\in\mathcal{P}_{\text{ac}}(\Omega)}\left\{\frac{1}{\varepsilon}\int_{\Omega}\left(f(\rho)+\left(\frac12-A\right)\rho-\rho\phi+\frac{1}{2}\phi^2\right)dx+\frac{\varepsilon}{2}\int_{\Omega}|\nabla\phi|^2dx\right\}.
\end{equation}
The construction of a solution to \eqref{eq:elliptic-parabolic-rescaled} satisfying the appropriate energy inequality was carried out in \cite{MR24} using the gradient flow structure of the equation. These results will be recalled in Proposition \ref{prop:well-posed-phi-epsilon}.

The study of the singular limit $\eps\to 0$ in \cite{MR24} (see also \cite{M23,KMW24}) relies on the following observation: 
A simple arithmetic manipulation (complete the square) allows us to rewrite the energy as follows:
\begin{equation}\label{eq:energy3}
\mathscr{J}_{\varepsilon}(\phi):=\inf_{\rho\in\mathcal{P}_{\text{ac}}(\Omega)}\left\{\frac{1}{\varepsilon}\int_{\Omega}\left(W(\rho)+\frac{1}{2}(\rho-\phi)^2\right)dx+\frac{\varepsilon}{2}\int_{\Omega}|\nabla\phi|^2dx\right\},
\end{equation}
where
\[
W(\rho):=f(\rho)+\left(\frac12-A\right)\rho-\frac{1}{2}\rho^2.
\]
This potential $W$ is defined for $\rho\geq 0$, but we will extend its definition to $\R$ by setting $W(\rho)=+\infty$ for $\rho<0$. The definition of $f$, \eqref{eq:f0},  leads to:
\[
W(\rho)=\begin{cases}
\left(\frac{1}{2}-A\right)\left(\rho-\rho^2\right) \qquad\qquad &\text{if}\quad\rho\in[0,1],\\
+\infty \qquad\qquad &\text{otherwise.}
\end{cases}
\]
\begin{figure}[tbp]
	\begin{center}
            \includegraphics[height=6cm]{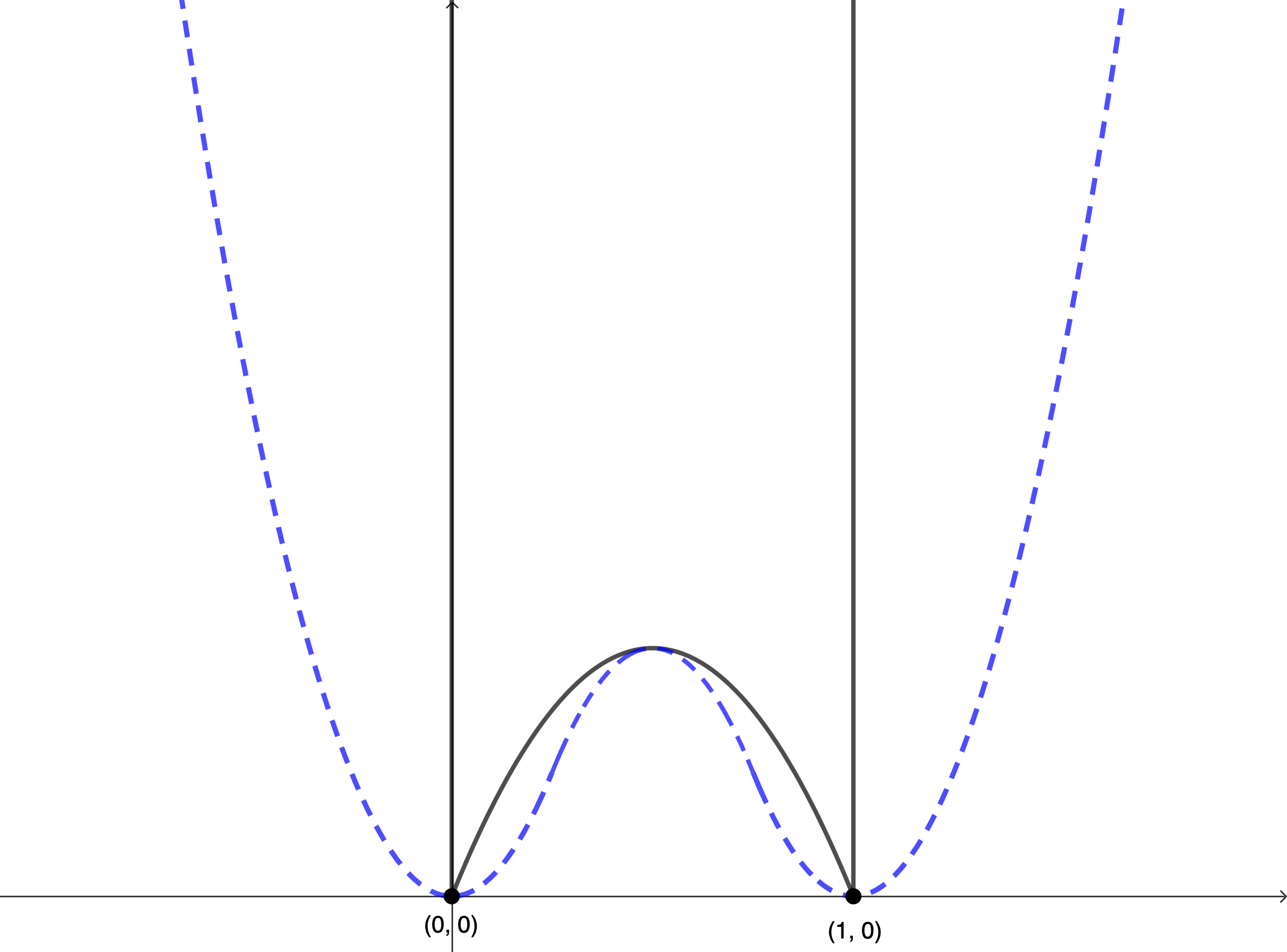}
		\vskip 0pt
		\caption{The double-well potentials $W(\rho)$ (solid line) and $\overline W(\phi)$ (dashed line).}
        \label{fig:double-well}
	\end{center}
\end{figure}
When $A\geq \frac  1 2$, this function $W$ is convex but when $A<\frac 1 2 $, it
is a double-well potential  (see Figure \ref{fig:double-well})  which satisfies
$$
W(\rho)\geq 0\quad  \mbox{ for all } \rho\in \R\,, \qquad W(\rho)=0 \Leftrightarrow \rho = 0 \mbox{ or } 1.
$$
From now on, we will thus assume that the nonlinearity $f(\rho)$ is given by \eqref{eq:f0} with
\begin{equation}\label{eq:A}
A\in\left(0,\frac 1 2\right).
\end{equation}
This ensures that the diffusion at low density is weaker than the attractive force when $\eps\ll1$ and is the key to observing  phase-separation phenomena (when $A\geq \frac1  2$ the limit $\eps\to0$ leads to constant densities).

We now see that the energy \eqref{eq:energy2} a structure reminiscent of the classical 
Allen-Cahn (or Modica-Mortola) functional, although the double-well potential acts on the density $\rho$ while the $H^1$ norm is that of the potential $\phi$.
We can make the connection a little bit more clear as follows:
First (to simplify the notations), we introduce the function $ g(\rho): = f(\rho) +\left(\frac12-A\right)\rho $, also extended by $+\infty$ for $\rho<0$.
The double-well potential can then  be  written as
$$ W(\rho) = g(\rho) -\frac 1 2\rho^2.$$
and we  introduce the dual potential
$$ \overline W(\phi) = \frac1  2 \phi^2 - g^\ast (\phi)$$
where $g^\ast $ denotes the Legendre transform\footnote{We recall that the Legendre transform is defined by
$$ g^\ast (\phi) = \sup_{\rho\in \R} \rho \phi - g(\rho).$$} of $g$.
With these notations, we can show the following:
\begin{proposition}\label{prop:AC}
\item[(i)] The function $\overline W$ is a $C^{1,1}$ double-well potential (see Figure \ref{fig:double-well}) satisfying 
$$
\overline W(\phi)\geq 0 \mbox{ for all } \phi \in \R, \qquad \overline W(\phi)=0 \Leftrightarrow \phi = 0 \mbox{ or } 1.
$$
\item[(ii)] The energy \eqref{eq:energy2} can be written as follows:
\begin{align}
\mathscr{J}_{\varepsilon}(\phi)
& :=
\inf_{\int \rho\, dx =1 }\left\{\frac{1}{\varepsilon}\int_{\Omega}g(\rho) + g^\ast (\phi) -\rho \phi \, dx\right\} +
\int_{\Omega}\frac{1}{\varepsilon}\overline W(\phi) + 
\frac{\varepsilon}{2} |\nabla\phi|^2dx.\label{eq:Jee}
\end{align}
\item[(iii)] For all $\phi\in L^2(\Omega)$,  the solution of the minimization problem  \eqref{eq:rhophi} is given by
$
\rho_\phi = {g^\ast}'(\phi+\ell) .
$
for some Lagrange multiplier $\ell=\ell(\phi)\in \R$.
\item[(iv)] Equation \eqref{eq:phiepsilon} can be rewritten as
\begin{equation}\label{eq:ACN}
\eps \pa_t \phi^\eps - \eps \Delta \phi ^\eps=-\eps^{-1}\overline W'(\phi^\eps) + \eps^{-1} \left( {g^\ast }'(\phi^\eps+\ell^\eps) - {g^\ast }'(\phi^\eps) \right)
\end{equation}
where $\ell^\eps(t)$ is a Lagrange multiplier defined by the constraint $\int_\Omega  {g^\ast }'(\phi^\eps(x,t)+\ell^\eps(t))\, dx =1$.
\end{proposition}
The proof of this proposition is straightforward and included in Appendix \ref{app:W} for the reader's sake.
For future reference, we note that we have the explicit formula
\begin{equation}\label{eq:gg}
g(\rho) := 
\begin{cases}
+\infty & \mbox{ if } \rho<0\\
A\rho^2 + (\frac 1 2 -A) \rho & \mbox{ if } \rho\in(0,1)\\
+\infty & \mbox{ if } \rho>1
\end{cases}
\end{equation}
and (using the definition of $g$ and the fact that $\phi \in \pa g(\rho) \Leftrightarrow \rho \in \pa g^\ast (\phi)$),
we find that 
$g^\ast $ is a $C^{1,1}$ function whose derivative is given by the piecewise linear function
\begin{equation}\label{eq:gp}
{g^\ast }' (\phi) = \begin{cases}
0 & \mbox{ if } \phi \leq \frac 1 2-A\\
\frac{1}{2A} ( \phi+A-\frac 1 2) & \mbox{ if } \frac 1 2-A\leq \phi \leq \frac 1 2+A \\
1 & \mbox{ if } \phi \geq \frac 1 2 +A.
\end{cases}
\end{equation}


In the form \eqref{eq:ACN}, it is now clear that our equation has the form of a Allen-Cahn type equation with a $C^{1,1}$ double-well potential $\overline W$ and a ``forcing" term 
$$G^\eps(t,x)=\eps^{-1}\left(  {g^\ast }'\left(\phi^\eps(t,x)+\ell^\eps(t)) - {g^\ast }'(\phi^\eps(t,x)\right)\right) $$ 
which arises because of the density constraint $\int \rho\, dx=1$ (this term vanishes when the Lagrange multiplier $\ell$ is zero).
Unlike the classical volume-preserving Allen-Cahn equation, this term does not directly enforce a condition on the integral of $\phi$, but asymptotically (when $\eps\ll1$) we will see that the effect is the same.
The singular  limit $\eps\to0$ of the Allen-Cahn equation with forcing term was studied in detailed in \cite{MR11} under the condition
$$\sup_{\eps>0} \int_0^T \int_\Omega \frac 1 \eps |G^\eps(t,x)|^2\, dx\, dt <\infty.$$ 
Proving that such a bound holds will be the most important step in our proof (see Proposition \ref{prop:L2-force}) and will requires a delicate control of the Lagrange multiplier $\ell^\eps(t)$.
But once this is proved, we will be able to apply well known techniques and derive the mean-curvature flow equation (and prove our main result Theorem \ref{thm:convergence-to-limit-flow}).

\subsection{$\Gamma$-convergence and conditional convergence result of \cite{MR24}}
All the terms in the energy \eqref{eq:Jee}
 are non-negative since  a function and its Legendre transform always satisfy
$$ g(\rho) + g^\ast (\phi) -\rho \phi \geq 0 \qquad \mbox{ for all } \rho,\phi \in \R$$
(with equality if and only if $\rho = {g^\ast }'(\phi)$).
So the first term in \eqref{eq:Jee}   controls the convergence of $\rho$ to ${g^\ast }'(\phi)$ while 
the second integral in \eqref{eq:Jee} is the classical Allen-Cahn energy (or Modica-Mortola functional) which is known to $\Gamma$-converge to the perimeter functional.
This gives the intuition behind the following proposition (see \cite{MR24,R25} for the proof):
\begin{proposition}[Theorem 2.3 in \cite{MR24}]
The functional $\mathscr{J}_{\varepsilon}$, given by the equivalent formulas \eqref{eq:energy2} or \eqref{eq:Jee},
$\Gamma$-converges with respect to the $L^1(\Omega)$-topology to the perimeter functional 
\begin{equation}\label{eq:per}
\mathscr J_0(\phi) = 
\begin{cases}
\displaystyle \gamma \int_\Omega |D\phi| =\gamma P(E) & \mbox{ if } \phi = \chi_E \in \mathrm{BV}(\Omega;\{0,1\})\\
+\infty & \mbox{ otherwise.}
\end{cases}
\end{equation}
with $\gamma:=\int_0^1\sqrt{2\overline{W}(s)}ds$.
\end{proposition}

This result will not play any role our  analysis, but it justifies heuristicaly  the presence of the mean-curvature operator in the limit $\eps\to0$. 
The derivation of the mean-curvature flow equation from \eqref{eq:phiepsilon} was established in \cite{MR24} under the assumption of convergence of the energy. More precisely, the main result of  \cite{MR24} states that given a subsequence  $\phi^{\eps_n}$ which converges to $\phi^0=\chi_{E(t)}$ (such a subsequence exists), if we assume that
\begin{align}\label{eq:energy-assumption}
\lim_{n\to\infty}\int_0^T\mathscr{J}_{\varepsilon_n}(\phi^{\eps_n}) dt=\int_0^T\mathscr{J}_{0}(\phi^0)dt,
\end{align}
then the evolution of the interface $\pa E(t)$ is described by the volume preserving mean-curvature flow \eqref{eq:VPMCF}.

Similar conditional results were obtained in related framework in \cite{KMW24,M23,MR24,R25} and in much broader contexts in \cite{JKM16,O98,EO15,LO16}. 
Since the $\Gamma$-convergence guarantees that
$$\liminf_{n\to\infty}\int_0^T\mathscr{J}_{\varepsilon_n}(\phi^{\eps_n}) dt\geq \int_0^T\mathscr{J}_{0}(\phi^0)dt,
$$
we see that \eqref{eq:energy-assumption} holds if and only if there is no loss of energy in the limit. Such energy losses typically happen due to overlapping or disappearing interfaces in the limit process.
In case of boundaries that overlap, for example, the left-hand side of \eqref{eq:energy-assumption} counts interfaces (at least) twice, while the right-hand side counts the interface only once. 

In this paper, we will prove an unconditional convergence result. However, because the phenomena described above cannot be discounted, we will rely on a weaker notion of solution than that used in \cite{MR24}:
Following the work of Brakke  \cite{B78} and Ilmanen \cite{I93} (see also \cite{MR08,MR11,T17,T23} for more recent work in this direction), we use a general notion of solutions using  \emph{integral varifolds}. 
 More precisely, we will characterize the evolution of certain {\it energy measures}.
 Ideally, we expect these measures to coincide with the surface area measures associated with the interface separating regions with density $0$ and $1$. However, in general these energy measures may be supported on hidden boundaries or 
take the multiplicity of the interface  into account.

As a final comment, we point out that the result of  \cite{MR24} was proved for a general nonlinearity $f(\rho)$, while we use the particular form of $f$ given by \eqref{eq:f0}.

\medskip

\section{Preliminaries and main results}\label{sec:main-results}
\subsection{Well-posedness}\label{subsec:well-posedness}
We recall the following proposition (proved in \cite{MR24} for smooth $f$ and in \cite{R25} for more general nonlinearities that include our function \eqref{eq:f0}):
\begin{proposition}\label{prop:well-posed-min-prob}
Suppose that $f$ is given by \eqref{eq:f0} and that \eqref{eq:Omega1} holds.
The following statements  hold:
\item[(i)] For all $\phi\in L^2(\Omega)$, 
there exists a unique  $\rho_{\phi}\in\mathcal{P}_{\textrm{ac}}(\Omega)$ solution of \eqref{eq:rhophi}. It satisfies $\|\rho_{\phi}\|_{L^{\infty}(\Omega)}\leq1$.
\item[(ii)] The map $\phi \mapsto \rho_\phi$ is Lipschitz from $L^2(\Omega)\to L^2(\Omega)$ with Lipschitz constant $A^{-1}$.
\end{proposition}
Next, we recall that
the gradient flow structure of  \eqref{eq:phiepsilon} was used in \cite[Section 4]{MR24}  to prove the existence of a solution when $f$ is a smooth function. For the type of  singular $f$ we are considering here this analysis was extended in \cite{R25}.
We thus state the following well-posedness result without proof:
\begin{proposition}\label{prop:well-posed-phi-epsilon}
Suppose that $f$ is given by \eqref{eq:f0} and that \eqref{eq:Omega1} holds. Then, for any $\varepsilon>0$ and $\phi_{in}^{\varepsilon}\in H^1(\Omega;\R_+)$, there exists a unique nonnegative strong solution $\phi^{\varepsilon}$ to \eqref{eq:phiepsilon} such that
\begin{align*}
\phi^{\varepsilon}\in H^1_{loc}(0,\infty;L^2(\Omega))\cap L^{\infty}_{loc}(0,\infty;H^1(\Omega))\cap L^2_{loc}(0,\infty;H^2(\Omega))
\end{align*}
and that the following energy dissipation inequality holds:
\begin{align}\label{eq:dissipation}
\mathscr{J}_{\varepsilon}(\phi^{\varepsilon}(T))+\frac 1 2 \int_0^T\int_{\Omega}\varepsilon|\partial_t\phi^{\varepsilon}|^2dxdt\leq\mathscr{J}_{\varepsilon}(\phi_{in}^{\varepsilon})\qquad\text{for any }T>0.
\end{align}
\end{proposition}

\medskip


\subsection{Basic definitions}\label{subsec:defns-GMT-main-thm}
Before we state our main result, we briefly recall some standard definitions from geometric measure theory  (see \cite{S83} for more details):

We first recall that a Radon measure $\mu$ on $\Omega$ is \emph{$k$-rectifiable} if $\int_{\Omega}\eta\,  d\mu=\int_{M}\eta\, \theta \, d\mathcal{H}^k$ for any $\eta\in C_c(\Omega)$ for some countably $k$-rectifiable and $\mathcal{H}^k$-measurable set $M\subset\Omega$ and a function $\theta\in L^1_{loc}(\mathcal{H}^k\lfloor M;\mathbb{R}_{>0})$. We then write $\mu=\theta\mathcal{H}^k\lfloor M$.
We call $\theta$ the multiplicity and say that $\mu$ is \emph{$k$-integral} if the multiplicity is integer-valued $\mathcal{H}^k$-a.e. on $M$.

For a Radon measure $\mu$ on $\Omega$ and $x\in\Omega$, if there exist a $(d-1)$-dimensional linear subspace $T\subset\mathbb{R}^d$ and a number $\Theta>0$ such that
$$
\lim_{r\to0}\frac{1}{\omega_{d-1}r^{d-1}}\int_{\Omega}\zeta\left(\frac{y-x}{r}\right)\,d\mu(y)=\Theta\int_T\zeta\, d\mathcal{H}^{d-1}\qquad\text{for all }\zeta\in C_c(\mathbb{R}^d),
$$
we say that $\mu$ has the approximate tangent plane $T$, denoted by $T_x\mu$. Here, $\omega_{d-1}$ is the Lebesgue volume of the unit ball in $\mathbb{R}^{d-1}$. If a Radon measure $\mu$ on $\Omega$ is $(d-1)$-rectifiable, represented by $\mu=\theta\mathcal{H}^{d-1}\lfloor M$ with the above notations, then $\mu$ has the approximate plane at $x\in M$ and $\Theta=\theta$ at $\mathcal{H}^{d-1}$-a.e. $x\in M$.

For a $(d-1)$-rectifiable Radon measure $\mu=\theta\mathcal{H}^{d-1}\lfloor M$ on $\Omega$ with a locally finite, $(d-1)$-rectifiable, and $\mathcal{H}^{d-1}$-measurable set $M\subset\Omega$ and $\theta\in L^1_{loc}(\mathcal{H}^k\lfloor M;\mathbb{R}_{>0})$, we say that $h$ is \emph{a generalized mean curvature vector} if
\begin{equation}\label{eq:mcv}
\int_{\Omega}\mathrm{div}_{T_x\mu}g \, d\mu=-\int_{\Omega}h\cdot g \, d\mu\qquad\text{for any }g\in C^1_c(\Omega;\mathbb{R}^d),
\end{equation}
where $T_xM$ is the approximate tangent space of $M$ at $x$.
We recall that for a $k$-dimensional subspace $S$ of $\mathbb{R}^d$ with an orthonormal basis $\{e_1,\cdots,e_k\}$ and for a vector field $g\in C^1_c(\mathbb{R}^d;\mathbb{R}^d)$, we have $\mathrm{div}_{S}g:=\sum_{i=1}^k e_i\cdot\nabla_{e_i}g$.

\medskip

The notion of weak solution considered in this paper is based on the ideas introduced in the seminal work of Brakke \cite{B78} and Ilmanen \cite{I93} and requires the introduction of  the notion of $L^2$-flow (see \cite{MR08}) which describes the evolution of integral varifolds with square integrable generalized mean-curvature and square integrable generalized velocity.

\begin{defn}[$L^2$-flow \cite{MR08}]\label{def:L2-flow}
Let $\{\mu_t\}_{t>0}$ be a family of $(d-1)$-rectifiable Radon measures on $\Omega$ such that $\mu_t$ has a generalized mean curvature vector $h\in L^2(\mu_t;\mathbb{R}^d)$ for a.e. $t\in(0,\infty)$. We call $\{\mu_t\}_{t>0}$ an $L^2$-flow if there exist a constant $C>0$ and a vector field $v\in L^2_{loc}\left(0,\infty;L^2(\mu_t)^d\right)$ such that
\[
v(t,x)\perp T_x\mu_t\quad\text{for $\mu$-a.e. }(x,t)\in\Omega\times(0,\infty)
\]
and
\[
\left|\int_0^{\infty}\int_{\Omega}(\eta_t+\nabla\eta\cdot v)\, d\mu_t\, dt\right|\leq C\|\eta\|_{C^0((0,\infty)\times\Omega)}
\]
for any $\eta\in C^1_c((0,\infty)\times\Omega)$. Here, $d\mu:=d\mu_tdt$ and $T_x\mu_t$ is the approximate tangent space of $\mu_t$ at $x$. A vector field $v\in L^2_{loc}\left(0,\infty;L^2(\mu_t)^d\right)$ satisfying the above conditions is called a generalized velocity vector.
\end{defn}
We point out that any generalized velocity is (on a set of good points) uniquely determined by the evolution $\{\mu_t\}_{t>0}$ (see \cite{MR08}).

\medskip

\subsection{Main results}
Given a solution $\phi^\eps(t,x)$ of \eqref{eq:phiepsilon}  (or, equivalently \eqref{eq:ACN}), 
we introduce the Radon measure on $\Omega$ that expresses the total energy associated  tothe Modica-Mortola functional:
\[\mu_{t}^{\varepsilon}(\varphi):=\frac{1}{\gamma}\int_{\Omega}\left(\frac{\overline W(\phi^{\varepsilon}(t,x))}{\varepsilon}+\frac{\varepsilon|\nabla\phi^{\varepsilon}(t,x)|^2}{2}\right)\varphi(x)\, dx\quad\ \text{ for any }\varphi\in C_c(\Omega)
\]
with
$$\gamma:=\int_0^1\sqrt{2\overline{W}(s)}ds.$$
Our main result is as follows:
\begin{theorem}\label{thm:convergence-to-limit-flow}
    Suppose $d=2, 3$ and  \eqref{eq:Omega1} holds. Let $f$ be given by \eqref{eq:f0}, \eqref{eq:A} and let $\phi^{\varepsilon}$ be the solution to \eqref{eq:phiepsilon} with initial data $\phi_{\textrm{in}}^{\varepsilon}$, where the initial data $(\phi_{\textrm{in}}^{\varepsilon})_{\varepsilon>0}\subset H^1(\Omega;\R_+)$ satisfy $\sup_{\varepsilon>0}\mathscr{J}_{\varepsilon}(\phi_{\textrm{in}}^{\varepsilon})< +\infty$. 
    
    Then, there exists a subsequence of $\varepsilon\to0$ (still denoted by $\varepsilon\to0$) such that
\begin{itemize}
            \item[(A)] There exists a function $ \phi^0\in C^{0,s}_{loc}\left([0,\infty);L^1(\Omega)\right)$ (for all $s\in\left(0,\frac{1}{2}\right)$) with
        \[
        \phi^0\in   L^{\infty}_{loc}\left(0,\infty;BV\left(\Omega;\left\{0,1\right\}\right)\right)\cap BV_{loc}\left((0,\infty)\times\Omega\right)
        \]
and such that
        \begin{itemize}
            \item[(A.i)] $\phi^{\varepsilon}\to\phi^0$ strongly in $C^{0,s}_{loc}\left([0,\infty);L^1(\Omega)\right)$ for any $s\in\left(0,\frac{1}{2}\right)$ and $\rho_{\phi^{\varepsilon}}\to\phi^0$ in $L^{\infty}(0,T;L^1(\Omega))$.
            \item[(A.ii)] $\phi^0 = \chi_{E(t)}$ with $E(t)\subset \Omega$ such that $P(E(t),\Omega)<+\infty$ and  $|E(t)|=1$ for all $t>0$.
        \end{itemize}
        \item[(B)] There exists a family of $(d-1)$-integral Radon measures $(\mu_t)_{t\geq0}$ on $\Omega$ such that
        \begin{itemize}
            \item[(B.i)]  $\mu^{\varepsilon}\to\mu$ as Radon measures on $\Omega\times[0,\infty)$, where $d\mu=d\mu_tdt$.
            \item[(B.ii)] $\mu_t^{\varepsilon}\to\mu_t$ as Radon measures on $\Omega$ for all $t\geq0$.
            \item[(B.iii)] $\|\nabla\chi_{E(t)}\|(\varphi)\leq\mu_t(\varphi)$ for any $t>0$ and $\varphi\in C(\Omega;\R_+)$. Thus, $\partial^{\ast}E(t)\subset\textrm{supp}(\mu_t)$ for any $t>0$.
        \end{itemize}
        \item[(C)] There exists $\Lambda \in L^2_{loc}(0,+\infty)$ and a 
               measurable function $\Theta:\partial^{\ast}E(t)\to\mathbb{N}$ such that $(\mu_t)_{t>0}$ is an $L^2$-flow with a generalized velocity vector such that 
               \begin{equation}\label{eq:MCF}
        v=h-\frac{\Lambda(t)}{\Theta(t,\cdot)}\nu(t,\cdot)\quad\text{$\mathcal{H}^{d-1}$-a.e. on }\partial^{\ast}E(t)
        \end{equation}
        where $h$ is the generalized mean-curvature vector defined by \eqref{eq:mcv} and 
   $\nu(t,\cdot)$ is the inner unit normal vector of $E(t)$ on $\partial^{\ast}E(t)$.

\end{itemize}
\end{theorem}

We make the following remarks:
\begin{itemize}
\item[(i)] We can show (see \cite[Proposition 4.5]{MR08} and \cite[Remark 3.2]{MR11}) that the projection $V:=v\cdot \left(\frac{\na\phi^0}{|\na\phi^0|}\right)=v\cdot\nu$ belongs to $L^1(|\na \phi^0|)$ and satisfies, for $\xi\in C^1_c((0,T)\times\Omega;\mathbb{R}^d)$
$$ \int_0^T\int_\Omega V(t,\cdot) \xi(t,\cdot)d|\na \phi^0(t,\cdot)|\, dt =- \int_0^T \int_\Omega \phi^0(t,x) \pa_t\xi(t,x) \, dx\,dt$$
so that $V$ coincides with the normal velocity as defined for example in \cite{MR24} and satisfies
\begin{align}\label{eq:normal-velocity}
V = h\cdot \nu(t,\cdot) -\frac{\Lambda(t)}{\Theta(t,\cdot)} \quad\text{$\mathcal{H}^{d-1}$-a.e. on }\partial^{\ast}E(t)    
\end{align}
where $h\cdot \nu(t,\cdot) $ is the scalar mean-curvature.

In particular, our solution can be compared to that obtained in \cite{MR24}: The main difference is the presence of the 
multiplicity $\Theta$. 
The advantage, of course, is that Theorem \ref{thm:convergence-to-limit-flow} is an unconditional result, while the result of  \cite{MR24} required the  energy convergence assumption \eqref{eq:energy-assumption} to derive \eqref{eq:MCF}.

\item[(ii)] We also note that Theorem  \ref{thm:convergence-to-limit-flow}  does not include any condition  on the contact angle between the free boundary $\pa E(t)$ and the fixed boundary $\pa\Omega$. 
The assumption \eqref{eq:energy-assumption} is used in \cite{MR24} to derive a normal contact angle condition, which we are unable to recover here. This is a classical difficulty with this more general approach.
For general contact angle conditions, still under \eqref{eq:energy-assumption}, we refer to \cite{HL24,KMW24}.
\end{itemize}

The Lagrange multiplier $\Lambda(t)$ will be obtained as the weak $L^2$ limit of 
\[
\Lambda^{\varepsilon}(t):=\frac{\ell^{\varepsilon}(t)}{\gamma \varepsilon} .\]
In fact, as in \cite{T17}, the most delicate part of the proof is the derivation of appropriate estimates on the Lagrange multiplier term. 
A key role in this paper is played by the following proposition (proved in Section~\ref{subsec:preliminaries}):
\begin{proposition}\label{prop:Lagrange}
  For all $T>0$, there exists $\varepsilon_0=\varepsilon_0(T,|\Omega|,\sup_{\varepsilon>0}\mathscr{J}_{\varepsilon}(\phi_{\textrm{in}}^{\varepsilon}))>0$ such that
        \[
        \sup_{\varepsilon\in(0,\varepsilon_0)}\int_0^T|\Lambda^{\varepsilon}(t)|^2dt<\infty.
        \]
        In particular, there exists $\Lambda(t) \in L^2_{loc}(0,\infty)$ such that
$$\Lambda^{\varepsilon}\rightharpoonup \Lambda\text{ weakly in }L^2(0,T)\qquad\text{for all }T>0.
        $$
\end{proposition}

\subsection{A brief review  of the literature}\label{subsec:literature}
The volume preserving mean-curvature flow \eqref{eq:VPMCF} is a classical free boundary problem. 
The existence of global-in-time classical solutions with convex smooth initial data was proved by Gage in \cite{G86} in 2-d and by Huisken in \cite{H87} in all dimensions. The short-time existence of classical solutions to \eqref{eq:VPMCF} with smooth initial data (possibly nonconvex) was proved by Escher and Simonett \cite{ES98}. Weak notions of solutions to \eqref{eq:VPMCF} have also been studied: In \cite{MSS16}, 
Mugnai, Seis and Spadaro proved 
 the global existence of weak solutions in the sense of distributions  in dimension $d\leq7$ using minimizing movement schemes (see also Luckhaus, Sturzenhecker \cite{LS95}) under an assumption similar to \eqref{eq:energy-assumption}. The convergence of thresholding schemes to distributional BV solutions of \eqref{eq:VPMCF} was proved by Laux and Swartz \cite{LS17} also under an assumption similar to \eqref{eq:energy-assumption}. 
A similar approach was used by Laux and Simon \cite{LS18} to prove the convergence of solutions of a multiphase mass-conserving Allen-Cahn equation toward the solution of a multiphase volume-preserving mean curvature flow.
The weak-strong uniqueness for gradient-flow calibrations was proved by Laux \cite{L24}.

Investigating sharp interface limits of \emph{Allen-Cahn}-type equations has been a important subject as well.
 The classical Allen-Cahn equation \cite{AC79} reads
\[
\partial_t\phi^{\varepsilon}=\Delta\phi^{\varepsilon}-\frac{1}{\varepsilon^2}W'(\phi^\varepsilon),
\]
with a double-well potential $W$ given by $W(s):=\frac{1}{2}(1-s^2)^2$. The fact that the transition layer of $\phi^{\varepsilon}$ converges to interfaces moving by mean curvature in the limit  $\varepsilon\to0$ was justified 
in particular by Bronsard and Kon \cite{BK91},
 Evans, Soner and Souganidis \cite{ESS92} (using viscosity solutions techniques - see also \cite{CGG91,ES91}) and  Ilmanen \cite{I93} using the notion of solutions introduced in Brakke's seminal work \cite{B78}. We also refer to Chen \cite{C96}
 and  references within for an extensive review of the  literature on the Allen-Cahn equation.
In \cite{Sa}, Sato gave a short proof of this convergence, using tools introduced by R\"oger and Sch\"atzle in \cite{RS06}.
In \cite{MR11}, Mugnai and R\"oger extended this analysis to equations that include a forcing term:
\[
\partial_t\phi^{\varepsilon}=\Delta\phi^{\varepsilon}-\frac{1}{\varepsilon^2}W'(\phi^\varepsilon) + G^\eps,
\]
using the concept of $L^2$-flow that they developed in \cite{MR08} (which we also  use in this paper, see Definition \ref{def:L2-flow}) and which is similar to the notion of Brakke's flow.

The volume preserving  Allen-Cahn equation suggested by Rubinstein and Sternberg in  \cite{RS92}, is given by
\begin{align}\label{eq:RS}
\partial_t\phi^{\varepsilon}=\Delta\phi^{\varepsilon}-\frac{1}{\varepsilon^2}W'(\phi^\varepsilon)+\frac{1}{\varepsilon}\lambda^{\varepsilon}(t),    
\end{align}
where $\lambda(t)$ is a Lagrange multiplier that enforces $\int \phi^{\varepsilon}(t,x)dx=\int \phi^{\varepsilon}(0,x)dx$. Radially symmetric solutions to \eqref{eq:RS} were first studied by Bronsard and Stoth in \cite{BS97}, and 
Chen, Hilhorst and Logak \cite{CHL10}    proved  the convergence of the solution to \eqref{eq:RS}  to the classical solution of \eqref{eq:VPMCF} under the assumption that the latter exists.

An alternative Allen-Cahn equation with volume conservation was suggested by Golovaty \cite{G97} (see also Bretin and Brassel \cite{BB11}):
\begin{align}\label{eq:golovaty}
\partial_t\phi^{\varepsilon}=\Delta\phi^{\varepsilon}-\frac{1}{\varepsilon^2}W'(\phi^\varepsilon)+\frac{1}{\varepsilon}\lambda^{\varepsilon}(t)\sqrt{2W(\phi^{\varepsilon})}.
\end{align}
In this model, the effect of the Lagrange multiplier is concentrated on the transition layer.
In  \cite{AA14}, Alfaro and Alifrangis proved that  the  solution of \eqref{eq:golovaty}  converges to the classical solution of \eqref{eq:VPMCF} if such a solution exists 
(Kroemer and Laux \cite{KL24} obtained a rate of the convergence under the same assumption).
In  \cite{T17},  Takasao  
proved the (unconditional) convergence for  \eqref{eq:golovaty} in dimension $d=2,3$ using the same 
%
notion of weak solutions of \eqref{eq:VPMCF} as in \cite{MR08}.
As far as we know, a similar result has not been obtained for \eqref{eq:RS} and 
the concentration of the Lagrange multiplier near the transition interface plays a  crucial role in the proof. 
In \cite{T23} (motivated by \cite{KK20,MSS16}) Takasao considered a slightly different choice of the multiplier which ensures a further regularity property, and obtained a similar convergence result in all dimensions.

Our equation bears some similarity with \eqref{eq:golovaty}, especially obvious in its form \eqref{eq:ACN}, although we point out that the volume preservation of the limit $\phi^0:=\lim_{\varepsilon\to0}\phi^{\varepsilon}$ comes from the mass preservation of $\rho^{\varepsilon}$ since the volume of $\phi^{\varepsilon}$ may not be preserved for a fixed $\varepsilon>0$. 
In \cite{MR24}, the second author and Rozowski showed the convergence of \eqref{eq:elliptic-parabolic-rescaled} to volume-preserving mean curvature flow \eqref{eq:VPMCF} under the assumption of convergence of the energy  \eqref{eq:energy-assumption}. 
The present paper establishes the first unconditional convergence results for this problem, in the spirit of \cite{MR11,T17}.

\section{Proof of Theorem \ref{thm:convergence-to-limit-flow}}\label{sec:pf-of-main-thm}
\subsection{Phase separation}\label{subsec:rho-equation}
The first step is to prove that phase separation takes  place in the limit $\eps\to0$.
The following proposition is essentially  Part (A) of Theorem \ref{thm:convergence-to-limit-flow}.
It's proof is similar to the proof of \cite[Proposition 6.1]{MR24} and is recalled here for the sake of completeness.
\begin{proposition}\label{prop:phase-separation}
Under the  assumptions of Theorem \ref{thm:convergence-to-limit-flow}, there exist a subsequence (still denoted by $\varepsilon\to0$) and a function $ \phi^0\in C^{0,s}_{loc}\left([0,\infty);L^1(\Omega)\right)$ (for all $s\in\left(0,\frac{1}{2}\right)$) with
\[
\phi^0\in   L^{\infty}_{loc}\left(0,\infty;BV\left(\Omega;\left\{0,1\right\}\right)\right)\cap BV_{loc}\left((0,\infty)\times\Omega\right)
\]
such that
\begin{itemize}
\item[(i)] $\phi^{\varepsilon}\to\phi^0$ strongly in $C^{0,s}_{loc}\left([0,\infty);L^1(\Omega)\right)$ for any $s\in\left(0,\frac{1}{2}\right)$ and $\rho_{\phi^{\varepsilon}}\to\phi^0$ in $L^{\infty}(0,T;L^1(\Omega))$.
\item[(ii)] $\phi^0 = \chi_{E(t)}$ with $E(t)\subset \Omega$ such that $P(E(t),\Omega)<+\infty$ and  $|E(t)|=1$ for all $t>0$.
\end{itemize}
\end{proposition}
\begin{proof}
Define the function $F:\R_+\to\mathbb{R}$ by $F(v):=\gamma^{-1}\int_0^v\sqrt{2\overline{W}(u)}du$ for $v\leq1$ and $F(v)=F(1)=1$ for $v>1$. We let $\psi^{\varepsilon}:=F(\phi^{\varepsilon})$. Then, the Lipschitz continuity of $F$ gives
\begin{align*}
|\psi^{\varepsilon}|\leq C|\phi^{\varepsilon}|\leq C(\rho_{\phi^{\varepsilon}}+|\phi^{\varepsilon}-\rho_{\phi^{\varepsilon}}|).
\end{align*}
We note that, by \eqref{eq:dissipation} and  \eqref{eq:energy3} we have
\begin{align*}
\int_{\Omega}|\phi^{\varepsilon}-\rho_{\phi^{\varepsilon}}|dx\leq|\Omega|^{1/2}\left(\int_{\Omega}|\phi^{\varepsilon}-\rho_{\phi^{\varepsilon}}|^2dx\right)^{1/2}\leq\left(|\Omega|\varepsilon\sup_{\varepsilon>0}\mathscr{J}_{\varepsilon}(\phi_{in}^{\varepsilon})\right)^{1/2}.
\end{align*}
Together with the mass constraint $\int_{\Omega}\rho_{\phi^{\varepsilon}}dx=1$, we thus have $\psi^{\varepsilon}\in L^{\infty}(0,\infty;L^1(\Omega))$. Also, by Young's inequality,
\begin{align*}
\int_{\Omega}|\nabla\psi^{\varepsilon}|dx\leq\int_{\Omega}\sqrt{2\overline{W}(\phi^{\varepsilon})}|\nabla\phi^{\varepsilon}|dx\leq\int_{\Omega}\frac{\varepsilon|\nabla\phi^{\varepsilon}|^2}{2}+\frac{W(\phi^{\varepsilon})}{\varepsilon}dx\leq\sup_{\varepsilon>0}\mathscr{J}_{\varepsilon}(\phi_{in}^{\varepsilon}).
\end{align*}
Therefore, $\psi^{\varepsilon}$ is bounded (uniformly in $\eps$) in  $L^{\infty}(0,\infty;BV(\Omega))$.


Next, we see that $\partial_t\psi^{\varepsilon}\in L^2(0,\infty;L^1(\Omega))$ since, for $T>0$,
\begin{align*}
\int_0^T\left(\int_{\Omega}|\partial_t\psi^{\varepsilon}|\right)^2dt&\leq\int_0^T\left(\int_{\Omega}\sqrt{2\overline{W}(\phi^{\varepsilon})}|\partial_t\phi^{\varepsilon}|\right)^2dt\\
&\leq\int_0^T\left(\int_{\Omega}\frac{2\overline{W}(\phi^{\varepsilon})}{\varepsilon}dx\right)\left(\int_{\Omega}\varepsilon|\partial_t\phi^{\varepsilon}|^2dx\right)dt\\
&\leq\sup_{\varepsilon>0}\mathscr{J}_{\varepsilon}(\phi_{in}^{\varepsilon})\int_0^T\int_{\Omega}\varepsilon|\partial_t\phi^{\varepsilon}|^2dx\\
&\leq\left(\sup_{\varepsilon>0}\mathscr{J}_{\varepsilon}(\phi_{in}^{\varepsilon})\right)^2.
\end{align*}
Since $BV(\Omega)$ is compactly embedded in $L^q(\Omega)$ for $q\in\left[1,\frac{d}{d-1}\right)$, a refined Aubin-Lions lemma (see \cite[Theorem 1.1]{A00}) implies that $\{\psi^{\varepsilon}\}_{\varepsilon>0}$ is relatively compact in $C^{0,s}([0,T];L^1(\Omega))$ for $s\in\left(0,\frac{1}{2}\right)$ and in $L^{\infty}(0,T;L^q(\Omega))$ for $q\in\left[1,\frac{d}{d-1}\right)$. Up to a subsequence, we can thus assume that $\psi^{\varepsilon}$ converges to some function $\phi^0$ strongly in $C([0,T];L^q(\Omega))$ as $\varepsilon\to0$ for $q\in\left[1,\frac{d}{d-1}\right)$. The lower-semicontinuity of the $BV$-seminorm yields $\phi^0\in L^{\infty}_{loc}(0,\infty;BV(\Omega))$ and $\phi^0\in BV_{loc}((0,\infty)\times\Omega)$.

\medskip

Now, we claim that
$$
\limsup_{\varepsilon\to0}\|\psi^{\varepsilon}-\phi^{\varepsilon}\|_{C([0,T];L^2(\Omega))}=0.
$$
Indeed, the function $v\to F(v)-v$ vanishes at $v=0,1$, which are precisely the zeroes of $\overline{W}$. Since $F$ is bounded, for each $\delta\in(0,1)$, there exists a constant $C_{\delta}>0$ such that
$$
|F(v)-v|\leq C_{\delta}\overline{W}(v)\qquad\text{for }v\notin V_{\delta},
$$
where $V_{\delta}:=\{v\in\R_+\,:\,\text{dist}(v,\{0,1\})<\delta\}$. Therefore, it holds (using \eqref{eq:dissipation}, \eqref{eq:energy2}) that
\begin{align*}
\int_{\Omega}\left|\psi^{\varepsilon}-\phi^{\varepsilon}\right|^2dx&=\int_{\{x:\phi^{\varepsilon}(x)\in V_{\delta}\}}\left|\psi^{\varepsilon}-\phi^{\varepsilon}\right|^2dx+\int_{\{x:\phi^{\varepsilon}(x)\notin V_{\delta}\}}\left|\psi^{\varepsilon}-\phi^{\varepsilon}\right|^2dx\\
&\leq C|\Omega|\delta^2+C_{\delta}\int_{\Omega}\overline{W}(\phi^{\varepsilon})dx\\
&\leq C|\Omega|\delta^2+C_{\delta}\varepsilon\sup_{\varepsilon>0}\mathscr{J}_{\varepsilon}(\phi_{in}^{\varepsilon}).
\end{align*}
Taking $\varepsilon\to0$ and $\delta\to0$ in turn, we obtain that $\limsup_{\varepsilon\to0}\|\psi^{\varepsilon}-\phi^{\varepsilon}\|_{C([0,T];L^2(\Omega))}=0.$ Therefore,  the convergence of $\psi^{\varepsilon}$ to $\phi^0$ implies that of  $\phi^\varepsilon$  $C([0,T];L^q(\Omega))$ for $q\in\left[1,\frac{d}{d-1}\right)$.
Furthermore Fatou's lemma implies
$$
0\leq\int_{\Omega}\overline{W}(\phi^0(t))dx\leq\liminf_{\varepsilon\to0}\int_{\Omega}\overline{W}(\phi^{\varepsilon}(t))dx\leq\liminf_{\varepsilon\to0}\left(\varepsilon\sup_{\varepsilon>0}\mathscr{J}_{\varepsilon}(\phi_{in}^{\varepsilon})\right)=0,
$$
and so the limit $\phi^0(t)$ is a characteristic function with values $0,1$  for each $t\in[0,T]$.
Writing $\phi^0(t)=\chi_{E(t)}$, we see that $P(E(t),\Omega)=\int_{\Omega}|\nabla\chi_{E(t)}|<+\infty$ since we already know $\phi^0(t)\in BV(\Omega)$. Finally, again by \eqref{eq:dissipation}, \eqref{eq:energy2}, we have
$$
\limsup_{\varepsilon\to0}\|\phi^{\varepsilon}-\rho_{\phi^{\varepsilon}}\|_{C([0,T];L^2(\Omega))}\leq\limsup_{\varepsilon\to0}\left(\varepsilon\sup_{\varepsilon>0}\mathscr{J}_{\varepsilon}(\phi_{in}^{\varepsilon})\right)^{1/2}=0.
$$
This in particular implies that $\rho_{\phi^{\varepsilon}}$ converges to $\phi^0$ in $C([0,T];L^q(\Omega))$ for $q\in\left[1,\frac{d}{d-1}\right)$, and that for each $t\in[0,T]$,
$$
|E(t)|=\int_{\Omega}\phi^0(t)dx=\lim_{\varepsilon\to0}\int_{\Omega}\rho_{\phi^{\varepsilon}}(t)dx=1
$$
which complete the proof.
\end{proof}

\subsection{$L^2$-estimate on the forcing term}\label{subsec:preliminaries}

The crucial estimate  that we need to establish  in order to use well-known machinery is  an $L^2$ bound on the  last term (the ``Lagrange multiplier" term) in our equation \eqref{eq:ACN} which we recall here for convenience:
\begin{equation}\label{eq:fr}
\eps \pa_t \phi - \eps \Delta \phi =-\eps^{-1}\overline W'(\phi) + \eps^{-1} \left( {g^\ast }'(\phi+\ell) - {g^\ast }'(\phi) \right)
\end{equation}
with $\ell=\ell(t)$ such that $\int_\Omega {g^\ast }'(\phi(t,x)+\ell(t))\, dx=1$.
The main result of this section and cornerstone of the paper  is the following:
\begin{proposition}\label{prop:L2-force}
Under the assumptions of Theorem \ref{thm:convergence-to-limit-flow}, there exists $\varepsilon_0=\varepsilon_0(T,|\Omega|,\sup_{\varepsilon>0}\mathscr{J}_{\varepsilon}(\phi_{\textrm{in}}^{\varepsilon}))>0$ such that
\[
\sup_{\varepsilon\in(0,\varepsilon_0)}\frac{1}{\varepsilon^3}\int_0^T\int_{\Omega}\left| {g^\ast }'(\phi^\eps+\ell^\eps(t)) - {g^\ast }'(\phi^\eps) \right|^2\, dx\, dt<\infty.
\]
\end{proposition}
The proof of this proposition relies on two facts: First, the fact that the rescaled Lagrange multiplier $\eps^{-1}\ell^\eps(t)$ can be controlled in $L^2(0,T)$ uniformly in space (it does not depend on $x$), and second,
 the concentration effect of the multiplier on the transition layer, which follows from the observation that 
 \[
{g^{\ast}}'(\phi+\ell)-{ g^{\ast}}'(\phi)=0\qquad\text{when  $\phi$ and $\phi+\ell$ are  near the wells of $\overline W$, $ 0$ and $1$},
\]
and allows us to use the energy to get some control in $L^2(\Omega)$ uniformly in time.

The first step is to establish the following $L^2$-estimate for the Lagrange multiplier
\[
\Lambda^{\varepsilon}(t):=\frac{1}{\gamma}\frac{\ell^{\varepsilon}(t)}{\varepsilon}.
\]
\begin{lemma}\label{lem:L2-multiplier}
Under the assumptions of Theorem \ref{thm:convergence-to-limit-flow}, there exists $\varepsilon_0=\varepsilon_0(T,|\Omega|,\sup_{\varepsilon>0}\mathscr{J}_{\varepsilon}(\phi_{\textrm{in}}^{\varepsilon}))>0$ such that
\[
\sup_{\varepsilon\in(0,\varepsilon_0)}\int_0^T|\Lambda^{\varepsilon}(t)|^2dt<\infty.
\]
\end{lemma}

\begin{proof}
We begin by multiplying \eqref{eq:fr} by $\varepsilon\nabla\phi^{\varepsilon}\cdot\xi$ 
where $\xi$ is a test vector field in $C^1([0,\infty)\times\overline{\Omega};\mathbb{R}^d)$ with $\xi\cdot n=0$ on $[0,\infty)\times\partial\Omega$.
We obtain
\begin{align*}
\int_{\Omega}\varepsilon\partial_t\phi^{\varepsilon}\nabla\phi^{\varepsilon}\cdot\xi \, dx=\int_{\Omega}\varepsilon\Delta\phi^{\varepsilon}\nabla\phi^{\varepsilon}\cdot\xi dx-\int_{\Omega}\varepsilon^{-1}\left( {g^\ast }'(\phi^\eps+\ell^\eps) - {g^\ast }'(\phi^\eps) \right) \nabla\phi^{\varepsilon}\cdot\xi\, dx.
\end{align*}
A couple of  integration by parts yield:
\[
\int_{\Omega}\varepsilon\Delta\phi^{\varepsilon}\nabla\phi^{\varepsilon}\cdot\xi \,dx=-\int_{\Omega}\varepsilon\nabla\phi^{\varepsilon}\otimes\nabla\phi^{\varepsilon}:D\xi \,dx+\int_{\Omega}\frac{\varepsilon}{2}|\nabla\phi^{\varepsilon}|^2\div \xi \, dx
\]
and 
\begin{align*}
\int_{\Omega}\varepsilon^{-1}\left( {g^\ast }'(\phi^\eps+\ell^\eps) - {g^\ast }'(\phi^\eps) \right) \nabla\phi^{\varepsilon}\cdot\xi \, dx
& = 
\int_{\Omega}\varepsilon^{-1}(- {g^\ast } (\phi^\eps+\ell^\eps)+{g^\ast }(\phi)) \div \xi \, dx.
\end{align*}
Since  $\rho_\phi  = {g^\ast }'(\phi+\ell)$,  the Fenchel-Young identity gives
$$
{g}(\rho_{\phi})+{g^\ast }(\phi+\ell) -\rho_{\phi}( \phi+\ell) =0. 
$$
We can thus write
\begin{align*}
\int_{\Omega}\varepsilon^{-1}\left( {g^\ast }'(\phi^\eps+\ell^\eps) - {g^\ast }'(\phi^\eps) \right) \nabla\phi^{\varepsilon}\cdot\xi\, dx
& = 
\int_{\Omega}\varepsilon^{-1}({g}(\rho_{\phi^\eps})  -\rho_{\phi^\eps}( \phi^\eps+\ell^\eps))+{g^\ast }(\phi^\eps)) \div \xi \, dx\\
& = 
\int_{\Omega}\varepsilon^{-1}({g}(\rho_{\phi^\eps})  -\rho_{\phi^\eps} \phi^\eps+{g^\ast }(\phi^\eps)) \div \xi dx  - \frac{\ell^\eps
}{\eps} \int_{\Omega} \rho_{\phi^\eps} \div \xi \, dx.
\end{align*}


Putting every together, we find the equality
\begin{align*}
\int_{\Omega}\varepsilon\partial_t\phi^{\varepsilon}\nabla\phi^{\varepsilon}\cdot\xi\, dx=&-\int_{\Omega}\varepsilon\nabla\phi^{\varepsilon}\otimes\nabla\phi^{\varepsilon}:D\xi \, dx+\int_{\Omega}\frac{\varepsilon}{2}|\nabla\phi^{\varepsilon}|^2\mathrm{div}\xi \, dx\\
&
-\varepsilon^{-1}\int_{\Omega} \big[ {g}(\rho_{\phi^\eps})  -\rho_{\phi^\eps} \phi^\eps+{g^\ast }(\phi^\eps)\big]\div\xi\, dx-\gamma \Lambda^\eps(t)\int_{\Omega}\rho_{\phi^{\varepsilon}}\div\xi \, dx.
\end{align*}
All the terms in this equality with the  exception of the last one, can be bounded by either the energy   \eqref{eq:Jee}  or the dissipation of energy 
$$ D^\eps(t) = \int_\Omega \eps |\pa_t\phi^\eps|^2\, dx.$$
In particular the definition of $\rho_\phi$ implies 
\begin{align*}
\varepsilon^{-1}\int_{\Omega}\big[ {g}(\rho_{\phi^\eps})  +{g^\ast }(\phi^\eps)-\rho_{\phi^\eps} \phi^\eps\big]\div\xi\, dx &  \leq 
\inf_{\int \rho\, dx =1 }\left\{\frac{1}{\varepsilon}\int_{\Omega}g(\rho) + g^\ast (\phi^\eps) -\rho \phi^\eps \, dx\right\} \| \div \xi\|_{L^\infty}\\
& \leq \| D\xi \|_{L^\infty(\Omega)} \mathscr J_\eps(\phi^\eps(t)).
\end{align*}
Proceeding similarly with the other terms, we deduce:
\begin{equation}\label{eq:bdL}
\gamma\Lambda^\eps(t) 
\int_{\Omega}\rho_{\phi^{\varepsilon}}(t)\div\xi\, dx 
\leq \| D\xi \|_{L^\infty(\Omega)} \mathscr J_\eps(\phi^\eps(t)) + \|\xi\|_{L^\infty} D^\eps(t)^{1/2} \mathscr J_\eps(\phi^\eps(t)) ^{1/2}.
\end{equation}
The rest of the proof follows that of \cite[Proposition 8.2]{MR24} and we give a brief sketch here for the reader's sake.
\medskip

In view of \eqref{eq:bdL}, we can get a $L^2$-bound on $\Lambda^{\varepsilon}(t)$ if we have  a uniform lower bound of $\int_{\Omega}\rho_{\phi^{\varepsilon}}(t)\div\xi dx$. Set $\eta:=1-|\Omega|^{-1}$, which is a positive number by \eqref{eq:Omega1}. For a fixed $t_0\in[0,T]$, we claim that we can find a vector field $\xi_{t_0}\in C^1(\overline{\Omega};\mathbb{R}^d)$ with $\xi_{t_0}\cdot n=0$ on $\partial\Omega$ such that
\begin{align*}
\int_{\Omega}\phi^0(t_0)\div(\xi_{t_0})dx\geq\frac{\eta}{2}.
\end{align*}
Indeed, we take a sequence of smooth functions $\psi_{k}\in C^{\infty}(\overline{\Omega})$ such that $\{\psi_{k}\}_{k}$ converges to $\phi^0(t_0,\cdot)-|\Omega|^{-1}$ strongly in $L^1(\Omega)$, $\int_\Omega \psi_kdx=0$ for all $k$, and $\sup_{k}\|\psi_k\|\leq 1$ and define $\xi_k=\nabla u_k$, where $u_k$ solves the Neumann boundary problem
\begin{align*}
\begin{cases}
\Delta u_k=\psi_k \quad &\text{in}\quad\Omega,\\
\nabla u_k\cdot n=0, \quad &\text{on}\quad\partial\Omega.
\end{cases}
\end{align*}
Then, we have $\div(\xi_k)=\psi_k\in C^1(\overline{\Omega})$ and $\xi_k\cdot n=0$ on $\partial\Omega$. Moreover, by Proposition \ref{prop:phase-separation}(ii),
\begin{align*}
\lim_{k\to\infty}\int_{\Omega}\phi^0(t_0,x)\div(\xi_k)(x)\, dx&=\lim_{k\to\infty}\int_{\Omega}\phi^0(t_0,x)\psi_k(x)\, dx\\
&=\int_{\Omega}\phi^0(t_0,x)\left(\phi^0(t_0,x)-|\Omega|^{-1}\right)\, dx=1-|\Omega|^{-1}=\eta.
\end{align*}
We can thus set $\xi_{t_0}:=\xi_{k_0}$ with a sufficiently large $k_0$ so that $\int_{\Omega}\phi^0(t_0,x)\div(\xi_{k_0})(x)\, dx\geq\frac{\eta}{2}$.

\medskip
Since $\phi^0\in C^{0,1/4}([0,T];L^1(\Omega))$ (see Proposition \ref{prop:phase-separation}(i)) and $\div(\xi_{t_0})\in L^\infty(\Omega)$, it also follows that 
\begin{align*}
\int_{\Omega}\phi^0(t)\div(\xi_{t_0})\, dx\geq\frac{\eta}{4}\qquad \mbox{ for $t\in(t_0-\delta,t_0+\delta)$ } 
\end{align*}
for some small $\delta>0$. 

Similarly, since $\rho_{\phi^{\varepsilon}}\to\phi^0$ in $L^{\infty}(0,T;L^1(\Omega))$ (see Proposition \ref{prop:phase-separation}(i)),
there exists $\varepsilon_{t_0}>0$ such that
\begin{align*}
\int_{\Omega}\rho_{\phi^{\varepsilon}}(t)\div(\xi_{t_0})\, dx\geq\frac{\eta}{8}
\end{align*}
for all $\varepsilon<\varepsilon_{t_0}$ and $t\in(t_0-\delta,t_0+\delta)$.

\medskip

By the compactness of the interval $[0,T]$, we can find $t_1,\dots , t_N$ such that  
$\cup_{i=1}^N(t_i-\delta_i,t_i+\delta_i) \supset [0,T]$.
We now define $\varepsilon_0=\min\{\varepsilon_{t_1},\cdots,\varepsilon_{t_N}\}$ and $K:=\max_{i=1,\cdots,N}\|\xi_{t_i}\|_{C^1(\overline{\Omega})}$.
Then, for every $\varepsilon<\varepsilon_0$ and $t\in[0,T]$, there exists $t_i$ such that $t\in (t_i-\delta_i,t_i+\delta_i)$ and 
by \eqref{eq:bdL} we get
\begin{align*}
\frac{\gamma\eta}{8}|\Lambda^{\varepsilon}(t)|\leq\gamma\left|\Lambda^{\varepsilon}(t)\int_{\Omega}\rho_{\phi^{\varepsilon}}(t)\div(\xi_{t_i})\, dx\right|\leq K\left(  \mathscr J_\eps(\phi^\eps(t)) + D^\eps(t)^{1/2} \mathscr J_\eps(\phi^\eps(t)) ^{1/2}\right).
\end{align*}
Since $\mathscr J_\eps(\phi^\eps(t)) ^{1/2} \in L^{\infty}(0,T)$ and $D^\eps(t)^{1/2}\in L^2(0,T)$ (see \eqref{eq:dissipation}), the result follows,
\end{proof}

\medskip



We now turn to the proof of Proposition \ref{prop:L2-force} which combines the previous estimates with a localisation property of  the function $x\mapsto {g^\ast}'(\phi(x)+\ell)-{g^\ast}'(\phi(x)) $. This last property requires in a critical way the very particular choice of nonlinearity \eqref{eq:f0} that we made in this paper.

\begin{proof}[Proof of Proposition \ref{prop:L2-force}]
We recall that our choice  of nonlinearity \eqref{eq:f0} and 
the definition of $g$  implies the following explicit formula for ${g^\ast}'$:
$$
{g^\ast}' (\phi) = \begin{cases}
0 & \mbox{ if } \phi \leq \frac 1 2-A\\
\frac{1}{2A} ( \phi+A-\frac 1 2) & \mbox{ if } \frac 1 2-A\leq \phi \leq \frac 1 2+A \\
1 & \mbox{ if } \phi \geq \frac 1 2 +A
\end{cases}
$$ 
In particular we note that ${g^\ast}'' (\phi)=0$ in $(-\infty,\delta)\cup (1-\delta,+\infty)$ with $\delta = \frac1 2-A>0$ and so
\begin{equation}\label{eq:zero}
 {g^\ast}'(\phi+\ell)-{g^\ast}'(\phi) = 0 \quad \mbox{ if } \phi \in \left(-\infty,\frac{\delta}{2}\right)\cup \left(1-\frac{\delta}{2},+\infty\right) \mbox{ and } |\ell|< \frac\delta 2.
 \end{equation}


For a fixed $T>0$, we define 
$$\mathcal N:=\left\{t\in[0,T]:|\ell^{\varepsilon}(t)|<\frac{1}{2}\delta\right\}.$$ 
Then for $t\in \mathcal N$, \eqref{eq:zero} gives:
\begin{align*}
\int_{\Omega}|{g^{\ast}}'(\phi^{\varepsilon}(x)+\ell^{\varepsilon})-{ g^{\ast}}'(\phi^{\varepsilon}(x))|^2dx&=\int_{\left\{x\in\Omega\ :\ \text{dist}(\phi^{\varepsilon}(x),\{0,1\})>\frac{\delta}{2}\right\}}|{g^{\ast}}'(\phi^{\varepsilon}(x)+\ell^{\varepsilon})-{ g^{\ast}}'(\phi^{\varepsilon}(x))|^2dx\\
&\leq |\ell^{\varepsilon}|^2\| {g^{\ast}}''\|_{L^\infty}^2\left|\left\{x\in\Omega:\text{dist}(\phi^{\varepsilon}(x),\left\{0,1\right\})>\frac{\delta}{2}\right\}\right|.
\end{align*}
The explicit formula for ${g^*}'$ above gives $\| {g^{\ast}}''\|_{L^\infty}= \frac{1}{2A}$ 
and the definition of $\mathscr J_\eps$ implies
\[
c_\delta  \left|\left\{x\in\Omega:\text{dist}(\phi^{\varepsilon}(x),\left\{0,1\right\})>\frac{\delta}{2}\right\}\right|\leq\int_{\left\{x\in\Omega\ :\ \text{dist}(\phi^{\varepsilon}(x),\{0,1\})>\frac{\delta}{2}\right\}} \overline {W}(\phi^{\varepsilon}) dx
 \leq  \eps \mathscr J_\eps(\phi^\eps(t) \leq  \eps \mathscr J_\eps(\phi^\eps_{in}).
\]
with 
$$c_\delta = \inf \left\{ \overline W(\phi)  , ;\, \text{dist}(\phi,\{0,1\})>\frac{\delta}{2}\right\}.$$
We thus have 
$$
\int_{\Omega}|{g^{\ast}}'(\phi^{\varepsilon}(x)+\ell^{\varepsilon})-{ g^{\ast}}'(\phi^{\varepsilon}(x))|^2dx
\leq  C |\ell^{\varepsilon}|^2\varepsilon = C |\Lambda^\eps(t)|^2 \eps^3.
$$
for some constant $C$ depending on $A$, $\delta$ and $ \mathscr J_\eps(\phi^\eps_{in})$.
Using Lemma \ref{lem:L2-multiplier} we deduce that
\begin{align}
\frac{1}{ \varepsilon^{3}} \int_{\mathcal N}\int_{\Omega}|{g^{\ast}}'(\phi^{\varepsilon}+\ell^{\varepsilon})-{ g^{\ast}}'(\phi^{\varepsilon})|^2dxdt\leq C\label{eq:int_N}    
\end{align}
for $\varepsilon\in(0,\varepsilon_0)$.

In order to control the integral over $[0,T]\setminus \mathcal N$ and complete the proof, we will first show that
\begin{equation}\label{eq:bvghjk}
\frac{1}{\eps}\int_{\Omega}|{g^{\ast}}'(\phi^{\varepsilon}(t)+\ell^{\varepsilon}(t))-{ g^{\ast}}'(\phi^{\varepsilon}(t))|^2dx \leq C \qquad\forall t\in [0,T].
\end{equation}
For this, we recall that $\rho_{\phi} = {g^\ast}'(\phi+\ell)$ minimizes $\int g(\rho) -\rho \phi\, dx$ with the constraint $\int \rho =1$, and we denote by $\bar \rho_{\phi} = {g^\ast}'(\phi)$ the minimizer of the same functional without any mass constraint.

The formulation of the energy \eqref{eq:energy3} gives
\begin{align*}
\mathscr J_\eps(\phi) 
& =
\inf_{\rho\in\mathcal{P}_{\text{ac}}(\Omega)}\left\{\frac{1}{\varepsilon}\int_{\Omega}\left(W(\rho)+\frac{1}{2}(\rho-\phi)^2\right)dx+\frac{\varepsilon}{2}\int_{\Omega}|\nabla\phi|^2dx\right\}\\
& = \frac{1}{\varepsilon}\int_{\Omega}\left(W(\rho_\phi)+\frac{1}{2}(\rho_\phi-\phi)^2\right)dx+\frac{\varepsilon}{2}\int_{\Omega}|\nabla\phi|^2dx
\end{align*}
which  implies
\begin{equation}\label{eq:inep1} \int_{\Omega} |\rho_{\phi^\eps}-\phi^\eps|^2\, dx = \int_{\Omega}  |{g^{\ast}}'(\phi^{\varepsilon}+\ell^{\varepsilon})-\phi^\eps|^2\, dx \leq 2 \mathscr J_\eps(\phi^\eps(t))  \eps\leq 2 \mathscr J_\eps(\phi^\eps_{in})  \eps
\end{equation}
Next, we note that we also have
\begin{align*}
\mathscr J_\eps(\phi) 
& \geq 
\inf_{\rho\geq 0} \left\{\frac{1}{\varepsilon}\int_{\Omega}\left(W(\rho)+\frac{1}{2}(\rho-\phi)^2\right)dx+\frac{\varepsilon}{2}\int_{\Omega}|\nabla\phi|^2dx\right\}\\
& = \frac{1}{\varepsilon}\int_{\Omega}\left(W(\bar \rho_\phi)+\frac{1}{2}(\bar \rho_\phi-\phi)^2\right)dx+\frac{\varepsilon}{2}\int_{\Omega}|\nabla\phi|^2dx.
\end{align*}
So we also  have 
\begin{equation}\label{eq:inep2} 
\int_{\Omega} |\bar \rho_{\phi^\eps}-\phi^\eps|^2\, dx = \int_{\Omega}  |{g^{\ast}}'(\phi^{\varepsilon} )-\phi^\eps|^2\, dx \leq 2 \mathscr J_\eps(\phi^\eps(t))  \eps\leq 2 \mathscr J_\eps(\phi^\eps_{in})  \eps
\end{equation}

Combining \eqref{eq:inep1} and \eqref{eq:inep2} gives \eqref{eq:bvghjk} which in turn implies 
$$
\frac{1}{\eps} \int_{[0,T]\setminus\mathcal N} \int_{\Omega}|{g^{\ast}}'(\phi^{\varepsilon}+\ell^{\varepsilon})-{ g^{\ast}}'(\phi^{\varepsilon})|^2dx\, dt \leq C |[0,T]\setminus\mathcal N|.
$$
Finally, Chebychev's inequality gives
$$
|[0,T]\setminus\mathcal N|
 = \left|\left\{ t\in [0,T]\, ;\, |\ell^\eps(t) | > \frac{\delta}{2}\right\}\right|
 \leq \frac{4}{\delta^2} \int_0^T |\ell^\eps(t)|^2\, dt =  \frac{C}{\delta^2} \eps^2 \int_0^T |\Lambda^\eps(t)|^2\, dt .
$$
Using Lemma \ref{lem:L2-multiplier} we deduce that
$$
\frac{1}{\eps^3} \int_{[0,T]\setminus\mathcal N} \int_{\Omega}|{g^{\ast}}'(\phi^{\varepsilon}+\ell^{\varepsilon})-{ g^{\ast}}'(\phi^{\varepsilon})|^2dx\, dt \leq  \frac{C}{\delta^2 } \int_0^T |\Lambda^\eps(t)|^2\, dt \leq C$$
for $\varepsilon\in(0,\varepsilon_0)$.

Together with \eqref{eq:int_N}, this implies the result.
\end{proof}

\subsection{End of the proof of Theorem \ref{thm:convergence-to-limit-flow}}\label{subsec:proof}
Thanks to the bound proved in Proposition \ref{prop:L2-force}, we can use well established tools for Allen-Cahn equation with forcing terms
 to complete the proof of  
Theorem \ref{thm:convergence-to-limit-flow}.
We summarize below the results that we will need, which can be found in 
\cite[Theorem 4.1, Proposition 4.9]{RS06} and
\cite[Theorem 3.1, Proposition 4.4]{MR11} (see also \cite[Theorem 4.1]{T17}).
\begin{theorem}\cite{RS06,MR11}\label{thm:MR11}
Suppose $d=2,3$ and let $\phi^{\varepsilon}$ be a solution of the perturbed Allen-Cahn equation
\begin{align}\label{eq:perturbed-allen-cahn}
\begin{cases}
\varepsilon\partial_t\phi^{\varepsilon}=\varepsilon\Delta\phi^{\varepsilon}-\frac{\overline{W}'(\phi^{\varepsilon})}{\varepsilon}+G^{\varepsilon} \quad &\text{in}\quad(0,T)\times\Omega,\\
\nabla\phi^{\varepsilon}\cdot n=0 \quad &\text{on}\quad(0,T)\times\partial\Omega, \\
\phi^{\varepsilon}(0,\cdot)=\phi^{\varepsilon}_{\textrm{in}} \quad &\text{in}\quad \Omega,
\end{cases}
\end{align}
where $\overline{W}\in C^{1}(\mathbb{R};\R_+)$ satisfies $\{\overline{W}=0\}=\{0,1\}$. We define
\[
 \mu_t^{\varepsilon}(\varphi):=\frac{1}{ {\gamma} }\int_{\Omega}\left(\frac{\overline{W}(\phi^{\varepsilon}(t,x))}{\varepsilon}+\frac{\varepsilon|\nabla\phi^{\varepsilon}(t,x)|^2}{2}\right)\varphi (x)\, dx, \qquad  \text{where } \gamma:=\int_0^{1}\sqrt{2\overline{W}(s)}\, ds
\]
for $t\in[0,T)$ and $\varphi \in C(\Omega)$. Suppose that there exists $\varepsilon_1>0$ such that
\begin{equation}\label{assthm}
\sup_{\varepsilon\in(0,\varepsilon_1)}\left(\mu_0^{\varepsilon}(\chi_{\Omega})+\frac{1}{\varepsilon} \int_0^T\int_{\Omega}|G^{\varepsilon}(x,t)|^2\, dx\, dt\right)<\infty.
\end{equation}
 Then, there exists a subsequence of $\varepsilon\to0$ (still denoted by $\varepsilon\to0$) such that
    \begin{itemize}
        \item[(a)] There exists a family of $(d-1)$-integral Radon measures $(\mu_t)_{t\in[0,T)}$ on $\Omega$ such that
        \begin{itemize}
            \item[(a.i)]  $\mu^{\varepsilon}\to\mu$ as Radon measures on $\Omega\times[0,T)$, where $d\mu=d\mu_tdt$.
            \item[(a.ii)] $\mu_t^{\varepsilon}\to\mu_t$ as Radon measures on $\Omega$ for all $t\in[0,T)$.
        \end{itemize}
        \item[(b)] There exists $G\in L^2\left(0,T;L^2(\mu_t)^d\right)$ such that
        \begin{equation}\label{eq:GG}
        \lim_{\varepsilon\to0}\frac{1}{ \gamma }\int_{\Omega\times[0,T)}-G^{\varepsilon}\nabla\phi^{\varepsilon}\cdot\xi\, dx\, dt=\int_{\Omega\times[0,T)}G\cdot\xi \, d\mu\quad\text{for any }\xi\in C_c\left([0,T)\times\Omega;\mathbb{R}^d\right).
        \end{equation}
        \item[(c)] $(\mu_t)_{t\in(0,T)}$ is an $L^2$-flow with a generalized velocity vector 
        $$v=h+G,$$ 
       and
        \[
        \lim_{\varepsilon\to0}\int_{\Omega\times[0,T)}v^{\varepsilon}\cdot\xi \, d\mu^{\varepsilon}=\int_{\Omega\times[0,T)}v\cdot\xi \, d\mu\quad\text{for any }\xi\in C_c\left([0,T)\times\Omega;\mathbb{R}^d\right).
        \]
        Here, $h$ is the generalized mean curvature vector of $\mu_t$ and
        \begin{align*}
        v^{\varepsilon}=
        \begin{cases}
        \frac{-\partial_t\phi^{\varepsilon}}{|\nabla\phi^{\varepsilon}|}\frac{\nabla\phi^{\varepsilon}}{|\nabla\phi^{\varepsilon}|} \quad &\text{if }\nabla\phi^{\varepsilon}\neq0,\\
        0 \quad &\text{otherwise.}
        \end{cases}
        \end{align*}
    \end{itemize}
\end{theorem}

\medskip

\begin{proof}[Proof of Theorem \ref{thm:convergence-to-limit-flow}]
We fix $T>0$ throughout the proof. Part (A) of Theorem \ref{thm:convergence-to-limit-flow} follows from Proposition \ref{prop:L2-force}.

\medskip

Next, we use the formulation \eqref{eq:ACN} of the equation to prove Parts (B) and (C): we see that 
 $\phi^{\varepsilon}$  solves \eqref{eq:perturbed-allen-cahn} with forcing term
$$G^{\varepsilon}:= \eps^{-1} \left( {g^\ast }'(\phi^\eps+\ell^\eps) - {g^\ast }'(\phi^\eps) \right) 
.$$
This term satisfies the following bound (see Proposition \ref{prop:L2-force}):
$$\sup_{\varepsilon\in(0,\varepsilon_0)}
\frac{1}{\varepsilon} \int_0^T\int_{\Omega}|G^{\varepsilon}(t,x)|^2\, dx\, dt <\infty
$$
for some $\varepsilon_0=\varepsilon_0(T,|\Omega|,\sup_{\varepsilon>0}\mathscr{J}_{\varepsilon}(\phi_{\textrm{in}}^{\varepsilon}))>0$. Finally, \eqref{eq:Jee} gives
$$ 
\mu_0^{\varepsilon}(\chi_{\Omega}) \leq \mathscr{J}_{\varepsilon}(\phi_{\textrm{in}}^{\varepsilon})).
$$
We thus see that under the assumption of Theorem \ref{thm:convergence-to-limit-flow}, the condition\eqref{assthm} is satisfied with $\varepsilon_1=\varepsilon_0$ and we can use Theorem \ref{thm:MR11}  to complete the proof of our main result:
\medskip

In particular, it is clear that Part (a) of Theorem \ref{thm:MR11} implies Part (B.i) and (B.ii) of Theorem \ref{thm:convergence-to-limit-flow} (the convergence of $\mu^\eps$ and $\mu^\eps_t$ as Radon measures). To prove (B.iii) of Theorem \ref{thm:convergence-to-limit-flow}, we define the function $F:\R_+\to\mathbb{R}$ by $F(v):=\gamma^{-1}\int_0^v\sqrt{2\overline{W}(u)}du$ for $v\leq1$ and $F(v)=F(1)=1$ for $v>1$, and we let $\psi^{\varepsilon}:=F(\phi^{\varepsilon})$. Then, for any $t>0$, we see that $\psi^{\varepsilon}\to\chi_{E(t)}$ as $\varepsilon\to0$ strongly in $L^1(\Omega)$ by Part (A) and by the Lipschitz continuity of $F$, and therefore by the semicontinuity,
\begin{align*}
\|\nabla\chi_{E(t)}\|(\varphi)&\leq\liminf_{\varepsilon\to0}\frac{1}{\gamma}\int_{\Omega}|\nabla\psi^{\varepsilon}(t)|\phi dx\\
&\leq\liminf_{\varepsilon\to0}\frac{1}{\gamma}\int_{\Omega}\left(\frac{\varepsilon|\nabla\phi^{\varepsilon}|^2}{2}+\frac{\overline{W}(\phi^{\varepsilon})}{\varepsilon}\right)\phi dx=\liminf_{\varepsilon\to0}\mu_t^{\varepsilon}(\varphi)=\mu_t(\varphi)
\end{align*}
for any $\varphi\in C(\Omega;\R_+)$. This implies  that $\partial^{\ast}E(t)\subset \textrm{supp}(\mu_t)$.

\medskip

We now turn to the proof of Part (C) of our theorem and the derivation of the mean-curvature flow equation.
First, we note that the $L^2(0,T)$ bound of Lemma \ref{lem:L2-multiplier} implies that (up to another subsequence) we can assume
$$
\Lambda^\eps(t) \to \Lambda(t)  \qquad \text{weakly in $L^2(0,T)$ as $\varepsilon\to0$}.
$$

Next, we need to identify the limiting force field $G$ defined by \eqref{eq:GG}.
This will  follow from the following lemma:
\begin{lemma}\label{lem:G}
We have 
\begin{align}\label{eq:limit-forcing-1}
\lim_{\varepsilon\to0}\frac{1}{\gamma}\int_{\Omega\times[0,T)}-G^\eps\nabla\phi^{\varepsilon}\cdot\xi \, dx\, dt=-\int_0^T\Lambda(t)\int_{\Omega}\xi\cdot\nu(t,\cdot) \, d\|\nabla\chi_{E(t)}\|\, dt 
\end{align}
for any $\xi\in C_c\left([0,T)\times\overline{\Omega};\mathbb{R}^d\right)$ with $\xi\cdot n=0$ on $[0,T)\times\partial\Omega$.
\end{lemma}
Postponing the proof of this lemma to the end of this section, we note that it implies (together with \eqref{eq:GG})
$$
G\, d\mu_t =- \Lambda(t) \nu(t,\cdot) \, d\|\nabla\chi_{E(t)}\| 
$$
From the absolute continuity of $\|\nabla\chi_{E(t)}\|$ with respect to $\mu_t$ proven in (B.iii) for a fixed $t>0$, there is a density function $\zeta:=\frac{d\|\nabla\chi_{E(t)}\|}{d\mu_t}$. We thus can write
$$
G(t,\cdot)=-\Lambda(t)\zeta(\cdot)\nu(t,\cdot)\qquad\mathcal{H}^{d-1}\text{-a.e.}\text{ on }\text{supp}(\mu_t). 
$$
On $\text{supp}(\mu_t)\setminus \partial^{\ast}E(t)$, we have $\zeta=0$, and thus, $G=0$. On $\partial^{\ast}E(t)$, we have $\zeta>0$, and thus, $\Theta:=\zeta^{-1}=\left(\frac{d\|\nabla\chi_{E(t)}\|}{d\mu_t}\right)^{-1}$ is defined, and
$$
G =- \frac{\Lambda(t) }{\Theta(t,\cdot)} \nu(t,\cdot)\qquad\mathcal{H}^{d-1}\text{-a.e.}\text{ on }\partial^{\ast}E(t). 
$$
Moreover, the function $\Theta$ on $\partial^{\ast}E(t)$ is integer-valued due to the integrality of $\mu_t$.

\medskip

Part (c) of Theorem \ref{thm:MR11} now yields the conclustion that $(\mu_t)_{t>0}$ is an $L^2$-flow with the generalized velocity vector
$$ v = h +G = h-\frac{\Lambda}{\Theta} \nu(t,\cdot) \quad \mathcal H^{d-1}\text{-a.e. on } \pa^*E(t).$$ 
\end{proof}

\begin{remark}
We can also prove that 
  \[
        \int_{\Omega}v\cdot\nu(t,\cdot)\, d\|\nabla\chi_{E(t)}\|=0\quad\text{for a.e. }t\in(0,\infty).
               \]
Indeed, \cite[Proposition 4.5]{MR08} gives
$$
\int_0^T \int_\Omega  v\cdot\nu(t,\cdot) \eta(t,x)\, d|\na \phi^0(t)|\, dt =- \int_0^T \int_\Omega  \phi^0\pa_t\eta\, dx\,dt
$$
for any $\eta \in C^1_c((0,T)\times \Omega)$.
In particular, for any functions $\zeta \in C^1_c((0,T))$ and $\varphi\in C^1_c(\Omega)$,
Part (A.ii) of Theorem \ref{thm:convergence-to-limit-flow} yields
\[
\int_0^T\zeta(t)\int_{\Omega} \varphi(x) v\cdot\nu \, d\|\nabla\chi_{E(t)}\|dt=-\int_0^T \zeta'(t)\int_{\Omega}\varphi(x)\chi_{E(t)}\, dx\, dt
\]
Taking a sequence of function $\varphi$ that converges to $1$ on $\Omega$, we deduce
\[
\int_0^T\zeta(t)\int_{\Omega}   v\cdot\nu \, d\|\nabla\chi_{E(t)}\|dt=-\int_0^T \zeta'(t)\int_{\Omega}\chi_{E(t)}\, dx\, dt
=- \int_0^T\zeta'(t)\,dt=0.
\]
The claim follows as the test function $\zeta\in C^1_c((0,T))$ was arbitrary.
\end{remark}

\medskip

\begin{proof}[Proof of Lemma \ref{lem:G}]
We write:
\begin{align*}
& \frac{1}{\gamma}\int_{\Omega\times[0,T)}-G^\eps\nabla\phi^{\varepsilon}\cdot\xi \, dx\, dt\\
&\qquad=\frac{1}{\gamma}\int_{\Omega\times[0,T]} - \varepsilon^{-1}\left({ g^{\ast}}'(\phi^{\varepsilon}(x)+\ell^\eps)-{g^\ast}'(\phi^{\varepsilon}(x) )\right)\nabla\phi^{\varepsilon}\cdot\xi \, dx\, dt\\
&\qquad=\frac{1}{\gamma}\int_{\Omega\times[0,T]}\varepsilon^{-1}\left(g^{\ast}(\phi^{\varepsilon}(x)+\ell^\eps)-g^{\ast}(\phi^{\varepsilon}(x))\right)\div \xi\,  dx\, dt\\
&\qquad=\frac{1}{\gamma}\int_{\Omega\times[0,T]}\varepsilon^{-1}\left(\int_0^{\ell^{\varepsilon}}{g^\ast}'(\phi^{\varepsilon}(x)+s)ds\right)\div \xi\,  dx\, dt\\
&\qquad=\frac{1}{\gamma}\int_{\Omega\times[0,T]}\varepsilon^{-1}\left(\int_0^{\ell^{\varepsilon}}{g^\ast}'(\phi^{\varepsilon}(x)+s)-{g^\ast}'(\phi^{\varepsilon}(x)+\ell^{\varepsilon})ds\right)\div \xi\, dx\, dt\\
&\qquad\qquad+\frac{1}{\gamma}\int_{\Omega\times[0,T]}\varepsilon^{-1} \ell^{\varepsilon}\rho_{\phi^{\varepsilon}}\div \xi\, dx\, dt\\
&=:I_1+I_2.
\end{align*}
The first integral $I_1$ goes to zero since
\begin{align*}
|I_1|\leq\frac{1}{\gamma} \varepsilon\| {g^\ast}''\|_{L^\infty} \int_{\Omega\times[0,T)}\left(\frac{\ell^{\varepsilon}(t)}{\varepsilon}\right)^2\div \xi \, dx\, dt= {\gamma}\frac{\varepsilon}{2A} \|\div \xi\|_{L^{\infty}([0,T)\times\Omega)}\int_0^T|\Lambda^{\varepsilon}(t)|^2dt
\end{align*}
which converges to zero when 
 $\varepsilon\to0$ thanks to the bound of Lemma \ref{lem:L2-multiplier}. The second integral $I_2$ converges to 
 $$\int_{\Omega\times[0,T]}\Lambda(t)\chi_{E(t)}\div \xi \, dx\, dt$$ since $\frac{\ell^{\varepsilon(t)}}{\gamma \varepsilon}  = \Lambda^\eps(t)\rightharpoonup \Lambda(t)$ weakly in $L^2(0,T)$ and
$$
\int_{\Omega}\rho_{\phi^{\varepsilon}}(t,x)\mathrm{div}(\xi(t,x))dx\to\int_{\Omega}\phi^0(t,x)\mathrm{div}(\xi (t,x))dx\quad\text{in }L^{\infty}(0,T),
$$
by Part (A.i) of Theorem \ref{thm:convergence-to-limit-flow}.
We have thus proved that
$$
\lim_{\eps\to 0}
 \frac{1}{\gamma}\int_{\Omega\times[0,T)}-G^\eps\nabla\phi^{\varepsilon}\cdot\xi \, dx\, dt\\
 = 
\int_{\Omega\times[0,T]}\Lambda(t)\chi_{E(t)}\div \xi \, dx\, dt
$$
and the result follows since $\na \chi_{E(t)} = \nu(t,x) \|\na \chi_{E(t)}\|$.


\end{proof}



\appendix

\section{Proof of Proposition \ref{prop:AC}}\label{app:W} 
\begin{proof}[Proof of Proposition \ref{prop:AC}]
\item[(i)] First, we note that we have the following alternate formula for $ \overline W(\phi)$:
\begin{equation}\label{eq:WW}
\overline W(\phi) = \inf_{\rho\in \R}\left\{W(\rho)+\frac{1}{2}(\rho-\phi)^2 \right\}
\end{equation}
Indeed, we have $W(\rho)+\frac{1}{2}(\rho-\phi)^2  = g(\rho) -\rho \phi +\frac 1 2\phi^2$ and so 
$$ 
\inf_{\rho\geq 0}\left\{W(\rho)+\frac{1}{2}(\rho-\phi)^2 \right\}
=\inf_{\rho\in \R}\left\{g(\rho) -\rho\phi+\frac{1}{2}\phi^2 \right\}
=  \frac{1}{2}\phi^2 - \sup_{\rho\in \R}\left\{ \rho \phi - g(\rho)\right\}
=  \frac{1}{2}\phi^2 - g^*(\phi).$$
Together with the fact that $W$ is a double well potential, \eqref{eq:WW} readily implies that $\overline W(\phi)\geq 0$ and $\overline W(\phi)=0$ iff $\phi=0$ or $\phi=1$.
The explicit formula \eqref{eq:gp} for ${g^\ast}'$ shows that $\overline W'$ is Lipschitz and so $\overline W$  is $C^{1,1}$.

\item[(ii)] Using \eqref{eq:energy2}, we now write
\begin{align*}
\mathscr{J}_{\varepsilon}(\phi) 
& = 
\inf_{\rho\in\mathcal{P}_{\text{ac}}(\Omega)}\left\{\frac{1}{\varepsilon}\int_{\Omega}\left(g(\rho)-\rho\phi+\frac{1}{2}\phi^2\right)dx+\frac{\varepsilon}{2}\int_{\Omega}|\nabla\phi|^2dx\right\} \\
& = 
\inf_{\rho\in\mathcal{P}_{\text{ac}}(\Omega)}\left\{\frac{1}{\varepsilon}\int_{\Omega}\left(g(\rho)+g^\ast(\phi) -\rho\phi+  \frac{1}{2}\phi^2-g^\ast (\phi) \right)dx+\frac{\varepsilon}{2}\int_{\Omega}|\nabla\phi|^2dx\right\} 
\end{align*}
and \eqref{eq:Jee} follows.

\item[(iii)] The minimization problem   \eqref{eq:rhophi} can be recast as
$$
\rho_\phi = \argmin \left\{ \int_\Omega g(\rho) - \rho\phi\, dx \, ;\, \int_\Omega \rho(x)\, dx=1\right\}
$$
Since $g$ is convex, classical arguments implies that 
there exists a Lagrange multiplier $\ell=\ell(\phi)\in \R$ such that
$ \ell \in \pa g(\rho_\phi) -\phi$
or equivalently $\rho_\phi = {g^\ast}'(\phi+\ell)$.

\item[(iv)] It is now straightforward to rewrite the right hand side of  \eqref{eq:phiepsilon}  as
$$ \rho_\phi - \phi = {g^\ast }'(\phi+\ell) - \phi  ={g^\ast }'(\phi+\ell) - {g^\ast }'(\phi) -\overline W'(\phi)$$
which leads to \eqref{eq:ACN}
\end{proof}

\bibliographystyle{plain}
\bibliography{Preprint}

\end{document}